\documentclass[11pt]{amsart}
\usepackage{amsmath, amssymb, amsthm, amsfonts, mathrsfs}
\usepackage{xcolor}

\newtheorem{theorem}{Theorem}[section]

\newtheorem{corollary}[theorem]{Corollary}
\newtheorem{proposition}[theorem]{Proposition}
\newtheorem{lemma}[theorem]{Lemma}

\theoremstyle{definition}

\DeclareSymbolFont{AMSb}{U}{msb}{m}{n}
\DeclareMathSymbol{\N}{\mathbin}{AMSb}{"4E}
\DeclareMathSymbol{\Z}{\mathbin}{AMSb}{"5A}
\DeclareMathSymbol{\R}{\mathbin}{AMSb}{"52}
\DeclareMathSymbol{\Q}{\mathbin}{AMSb}{"51}
\DeclareMathSymbol{\I}{\mathbin}{AMSb}{"49}
\DeclareMathSymbol{\C}{\mathbin}{AMSb}{"43}

\theoremstyle{definition}

\newcommand{\Aut}{\mbox{\rm Aut}}
\newcommand{\xmax}{x_{\scriptsize\rm max}}
\newcommand{\xmin}{x_{\scriptsize\rm min}}
\newcommand{\toprank}{\mbox{\rm rank}_{\scriptsize\rm top}}
\newcommand{\symbrank}{\mbox{\rm rank}_{\scriptsize\rm symb}}
\newcommand{\diam}{\mbox{\rm diam}}
\newcommand{\rest}{\!\upharpoonright\!}
\newcommand{\Ret}{\mbox{\rm Ret}}
\begin{document}
\title[Subshifts of Finite Symbolic Rank]{Subshifts of Finite Symbolic Rank}


\author{Su Gao}
\address{School of Mathematical Sciences and LPMC, Nankai University, Tianjin 300071, P.R. China}
\email{sgao@nankai.edu.cn}
\thanks{The first author acknowledges the partial support of his research by the National Natural Science Foundation of
China (NSFC) grants 12250710128 and 12271263.
}

\author{Ruiwen Li}
\address{School of Mathematical Sciences and LPMC, Nankai University, Tianjin 300071, P.R. China}
\email{rwli@mail.nankai.edu.cn}


\date{\today}


\begin{abstract} The definition of subshifts of finite symbolic rank is motivated by the finite rank measure-preserving transformations which have been extensively studied in ergodic theory. In this paper we study subshifts of finite symbolic rank as essentially minimal Cantor systems. We show that minimal subshifts of finite symbolic rank have finite topological rank, and conversely, every minimal Cantor system of finite topological rank is either an odometer or conjugate to a minimal subshift of finite symbolic rank. We characterize the class of all minimal Cantor systems conjugate to a rank-$1$ subshift and show that it is dense but not generic in the Polish space of all minimal Cantor systems. Within some different Polish coding spaces of subshifts we also show that the rank-1 subshifts are dense but not generic. Finally we study topological factors of minimal subshifts of finite symbolic rank. We show that every infinite odometer and every irrational rotation is the maximal equicontinuous factor of a minimal subshift of symbolic rank $2$, and that a subshift factor of a minimal subshift of finite symbolic rank has finite symbolic rank. 
\end{abstract}

\maketitle

\section{Introduction}

This paper is a contribution to the study of symbolic and topological dynamical systems, but the main notion studied here originated from ergodic theory. 

One of the main sources of examples and counterexamples in ergodic theory has been the measure-preserving transformations constructed from a cutting-and-stacking process. The first such example was given by Chac\'{o}n in \cite{Ch} more than half a century ago, and since then there has been a large literature devoted to the study of measure-theoretic properties of these transformations. In fact, much of the work concentrated on the so-called {\em rank-one transformations}, where there is only one stack in every step of the cutting-and-stacking process. This is partially because the class of rank-one transformations forms a dense $G_\delta$ set in (the Polish space of) the class of all measure-preserving transformations (c.f., e.g. \cite{Fe} and \cite{GH14}), and thus behaviors of the rank-one transformations capture the generic behaviors of all measure-preserving transformations. On the other hand, measure-preserving transformations of arbitrary finite rank (where there is a uniform finite bound on the number of stacks used in every step of the cutting-and-stacking process, first defined by Ornstein--Rudolph--Weiss in \cite{ORW}) have also been extensively studied. In particular, they are known to have different models, some of which are of geometric nature, and some symbolic. Ferenczi \cite{Fe} gave an excellent survey over a quarter of a century ago.

Meanwhile, there has also been an effort to develop an analogous theory of topological dynamics for Cantor systems of finite topological rank. The starting point was \cite{HPS} by Herman--Putnam--Skau, in which essentially minimal Cantor systems were described by their Bratteli--Vershik representations through a nested sequence of Kakutani--Rohlin partitions. These were seen to be analogous/parallel to the cutting-and-stacking processes used to build measure-preserving transformations. In the time period between \cite{HPS} and \cite{DM} by Downarowicz--Maass, many special kinds of Cantor systems were studied through their Bratteli--Vershik representations. These systems included odometers, substitution subshifts, linear recurrent subshifts, symbolic codings of interval exchange maps, etc. Along with these studies, there had been a number of attempts to produce a good definition of topological rank for Cantor systems. In the end, the natural notion of finite rank Bratteli diagram (where there is a uniform bound on the number of nodes on all of its levels) was chosen to be the connotation of finite topological rank. The supporting evidence was abundant. First, all the special kinds of systems mentioned above have finite rank Bratteli diagrams. More importantly, in \cite{DM} the authors considered general Cantor minimal systems with finite rank Bratteli diagrams and proved that they are either equicontinuous or else expansive and hence conjugate to subshifts. This important dichotomy theorem and its proof suggested that the assumption of finite rank Bratteli diagram has far-reaching consequences. Shortly after, the terminology of finite topological rank started to be used (by Durand in \cite{Du}), and since then the Cantor minimal systems of finite topological rank have been more extensively studied in this generality (see Durand \cite{Du}, Bressaud--Durand--Maass \cite{BDM}, Bezuglyi--Kwiatkowski--Medynets--Solomyak \cite{BKMS}, Donoso--Durand--Maass--Petite \cite{DDMP}, Durand--Perrin \cite{DP} and Golestani--Hosseini \cite{GH}).

At the same time, motivated by the symbolic definition of measure-preserving transformations of finite rank, Ferenczi \cite{Fe96} introduced a notion of $\mathcal{S}$-adic subshift. An $\mathcal{S}$-adic subshift is defined from a substitution process with infinitely many levels, and it has finite alphabet rank if there is a bound on the numbers of letters used on all of the levels. This notion of rank, as well as its interplay with the notion of finite topological rank, have been studied in Durand \cite{Du00},  Berth\'e--Delecroix \cite{BD14}, Leroy \cite{Le}, Berth\'e--Steiner--Thuswaldner--Solomyak \cite{BSTY}, Donoso--Durand--Maass--Petite \cite{DDMP}, Espinoza \cite{Es} and \cite{Es2023} etc. In particular, it has been  shown in \cite{DDMP} that every minimal Cantor system of finite topological rank is either an odometer or conjugate to an $\mathcal{S}$-adic subshift of finite alphabet rank. Conversely, every $\mathcal{S}$-adic subshift of finite alphabet rank has finite topological rank. 

In this paper we consider a notion of symbolic rank for subshifts which is more directly motivated by the symbolic definition of finite rank measure-preserving transformations. In some sense, a subshift of finite symbolic rank is simply a finite rank measure-preserving transformation without the measure. 
Rank-one subshifts in this sense have been studied in the literature. For example, they are known to have zero topological entropy, and Bourgain \cite{Bo} proved Sarnak's conjecture for minimal rank-one subshifts. Other topological properties of rank-one subshifts have been considered in Adams--Ferenczi--Petersen \cite{AFP}, Danilenko \cite{Da}, El Abdalaoui--Lema\'nczyk--de la Rue \cite{ELD}, Etedadialiabadi--Gao \cite{EG}, Gao--Hill \cite{GH} and \cite{GH16}, Gao--Ziegler \cite{GZ19} and \cite{GZ20}, etc.

A systematic study of subshifts of finite symbolic rank started in \cite{GJJLLSW}. 
Among other things, it was proved that they all have zero topological entropy. Much of \cite{GJJLLSW} focused on the combinatorial properties of infinite words that generate subshifts of finite symbolic rank, and not so much on the topological properties of the subshifts themselves. In particular, one of the main questions left unsolved was how the symbolic rank relates to the other more well-established notions of rank for various Cantor systems. 



In this paper we prove a number of results on the topological properties of subshifts of finite symbolic rank. The subshifts we consider all have alphabet $\{0,1\}$, and therefore they are subshifts of $2^\mathbb{Z}$. We show that any minimal subshift of finite symbolic rank has finite topological rank (Theorem~\ref{thm:6.7}) and conversely, any Cantor minimal system of finite topological rank is either an odometer or conjugate to a minimal subshift of finite symbolic rank (Theorem~\ref{thm:6.9}). 
Thus the notion of finite symbolic rank is essentially the same as the notion of finite topological rank, and by previous results (\cite{DDMP}), it is also essentially the same as the notion of finite alphabet rank. 

Although we do not solve any open problem about systems of finite topological rank in this paper, it is our hope that the relatively new notion of finite symbolic rank will provide a new perspective and an alternative approach to the future study of these systems. For example, one of the key open problems about minimal Cantor systems is their classification up to topogical conjugacy. In \cite{GH} the authors have described a complete classification of subshifts of symbolic rank one. This gives hope that further progress can be made by considering subshifts of higher symbolic ranks.  Another example of a major open problem is Sarnak's conjecture for topological dynamical systems of zero topological entropy. Bourgain \cite{Bo} gave a proof for all minimal subshifts of symbolic rank one, and in \cite{EG} Sarnak's conjecture was confirmed for a class of subshifts of symbolic rank one which is essentially minimal but not minimal. Note that by results of \cite{AD} or our Proposition~\ref{prop:6.11} below, subshifts of symbolic rank one can have arbitrarily high topological rank or alphabet rank. Hence the results of \cite{GH}, \cite{Bo} and \cite{EG} cannot be covererd by results about any fixed topological rank or alphabet rank.

In this paper we also consider various classes of Cantor systems and characterize their descriptive complexity. In particular, the class of all Cantor systems can be coded by the Polish space $\Aut(\mathcal{C})$ (this is defined and discussed in Subsection~\ref{subsec:2.3}), and we show that the classes of all essentially minimal Cantor systems, minimal Cantor systems, as well as those whose topological rank has a fixed bound all form $G_\delta$ subspaces, and hence are Polish (Section~\ref{sec:3}). On the other hand, the class of all minimal Cantor systems conjugate to a rank-$1$ subshift is dense but not generic in the Polish space of all minimal Cantor systems (Proposition~\ref{prop:5.4}). Additionally we consider two more Polish spaces of subshifts as done in \cite{PS} and show that the class of minimal subshifts conjugate to a rank-1 subshift is dense in these spaces but is not generic in either of them. This is in contrast with the situation in the measure-theoretic setting. Nevertheless, together with the results of \cite{PS}, our results show that the class of minimal subshifts conjugate to one of symbolic rank $\leq 2$ is generic in both of these Polish coding spaces of subshifts.

We also consider topological factors of minimal subshifts of finite symbolic rank (Section~\ref{sec:7}). We improve Theorem~\ref{thm:6.9} by showing that a minimal subshift of finite topological rank $\geq 2$ must be of finite symbolic rank itself (Corollary~\ref{cor:7.4}), and is not just conjugate to a subshift of finite symbolic rank as guaranteed by Theorem~\ref{thm:6.9}. However, the symbolic rank of the subshift might be much greater than the one to which it is conjugate. We show that any infinite odometer and any irrational rotation is the maximal equicontinuous factor of a minimal subshift of symbolic rank $2$, which is in contrast with known results about rank-$1$ subshifts.

The rest of the paper is organized as follows. In Section~\ref{sec:2} we give the preliminaries on descriptive set theory, topological dynamical systems, (essentially) minimal Cantor systems, ordered Bratteli diagrams, Kakutani--Rohlin partitions, subshifts, and what it means for a subshift to have finite symbolic rank. In Section~\ref{sec:3} we compute the descriptive complexity of the classes of essentially minimal Cantor systems and those with topological rank $\leq n$ for some $n\geq 1$, by giving some topological characterizations of these classes within the Polish space of all Cantor systems. In Section~\ref{sec:5} we give a topological characterization of all minimal Cantor systems conjugate to a rank-$1$ subshift. 
In Section~\ref{sec:4} we characterize minimal subshifts of finite symbolic rank as exactly those admitting a proper finite rank construction with bounded spacer parameter. This will be a basic tool in the study of minimal subshifts of finite symbolic rank. Section~\ref{sec:6} is the main section of this paper, in which we prove the main theorems (Theorems~\ref{thm:6.7} and \ref{thm:6.9}) which clarify the relationship between the notions of symbolic rank and topological rank. We give some examples to show that our results are in some sense optimal. We also prove a result connecting the notion of finite alphabet rank for $\mathcal{S}$-adic subshifts with the notion of finite symbolic rank. This gives an alternative proof of Theorem~\ref{thm:6.9} via the main result of \cite{DDMP}. In Section~\ref{sec:new6} we consider the density and the genericity of the class of all minimal subshifts conjugate to a rank-$1$ subshift in various Polish coding spaces of Cantor systems and subshifts. Finally, in Section~\ref{sec:7} we consider topological factors of minimal subshifts of finite symbolic rank. 

{\em Acknowledgments.} We thank Fabien Durand, Felipe Garc\'{\i}a-Ramos, Samuel Petite, and Todor Tsankov for useful discussions on the topics of this paper. Particularly, we thank Samuel Petite for pointing out reference \cite{Es} to us and Felipe Garc\'{\i}a-Ramos for bringing reference \cite{PS} to our attention. We also thank the anonymous referee for useful suggestions which improved the paper and for pointing out reference \cite{DDM2000}.

\section{Preliminaries\label{sec:2}}

\subsection{Descriptive set theory}
In the rest of the paper we will be using some concepts, terminology and notation from descriptive set theory. In this subsection we review these concepts, terminology and notation, which can be found in \cite{Ke}. 

A {\em Polish space} is a topological space that is separable and completely metrizable. 

Let $X$ be a Polish space and $d_X$ be a compatible complete metric on $X$. Let $K(X)$ be the space of all compact subsets of $X$, and let $d_H$ be the {\em Hausdorff metric} defined on $K(X)$ as follows. For $A\in K(X)$ and $x\in X$, let $d(x,A)=\inf\{d(x,y)\,:\, y\in A\}$. Now for $A, B\in K(X)$, let
$$ d_H(A,B)=\max\left\{ \sup\{d(x, B)\,:\, x\in A\},\sup\{d(y,A)\,:\, y\in B\}\right\}. $$
Then $d_H$ is a metric on $K(X)$ that makes $K(X)$ a Polish space. Moreover, if $X$ is compact, then $K(X)$ is compact. 

Let $X$ be a Polish space. A subset $A$ of $X$ is {\em $G_\delta$} if $A$ is the intersection of countably many open subsets of $X$. A subspace $Y$ of $X$ is Polish iff $Y$ is a $G_\delta$ subset of $X$. We say that a subset $A$ of $X$ is {\em generic}, or the elements of $A$ are {\em generic} in $X$, if $A$ contains a dense $G_\delta$ subset of $X$.

More generally, by a transfinite induction on $1\leq \alpha<\omega_1$, we can define the {\em Borel hierarchy} on $X$ as follows:
$$\begin{array}{rcl}
{\bf\Sigma}^0_1&=& \mbox{ the collection of all open subsets of $X$} \\
{\bf\Pi}^0_1&=& \mbox{ the collection of closed subsets of $X$} \\
{\bf\Sigma}^0_\alpha&=& \left\{ \bigcup_{n\in\mathbb{N}} A_n\,:\, A_n\in {\bf\Pi}^0_{\beta_n} \mbox{ for some $\beta_n<\alpha$}\right\} \\
{\bf\Pi}^0_\alpha &=& \left\{ X\setminus A\,:\, A\in {\bf\Sigma}^0_\alpha\right\}
\end{array}
$$
We also define ${\bf\Delta}^0_\alpha={\bf\Sigma}^0_\alpha\cap{\bf\Pi}^0_\alpha$. Thus ${\bf\Delta}^0_1$ is the collection of all clopen subsets of $X$. With this notation, $\bigcup_{\alpha<\omega_1}{\bf\Sigma}^0_\alpha=\bigcup_{\alpha<\omega_1}{\bf\Pi}^0_\alpha=\bigcup_{\alpha<\omega_1}{\bf\Delta}^0_\alpha$ is the collection of all {\em Borel} subsets of $X$. The collection of all $G_\delta$ subsets of $X$ is exactly ${\bf\Pi}^0_2$.

Let $X$ be a topological space. Recall that a subset $A$ of $X$ is {\em nowhere dense} in $X$ if the interior of the closure of $A$ is empty. $A$ is {\em meager} in $X$ if $A\subseteq \bigcup_{n\in\mathbb{N}} B_n$ where each $B_n$ is nowhere dense in $X$. $A$ is {\em nonmeager} in $X$ if it is not meager in $X$; $A$ is {\em comeager} in $X$ if $X\setminus A$ is meager in $X$.

\subsection{Topological dynamical systems}
The concepts we review in this subsection are standard and can be found in any standard text on topological dynamics, e.g. \cite{Au} and \cite{Ku}. By a {\em topological dynamical system} we mean a pair $(X, T)$, where $X$ is a compact metrizable space and $T: X\to X$ is a homeomorphism. If $(X, T)$ is a topological dynamical system and $Y\subseteq X$ satisfies $TY=Y$, then $Y$ is called a {\em $T$-invariant} subset. 

If $(X, T)$ and $(Y, S)$ are topological dynamical systems and $\varphi: X\to Y$ is a continuous surjection satisfying $\varphi\circ T=S\circ \varphi$, then $\varphi$ is called a {\em factor map} and $(Y, S)$ is called a ({\em topological}) {\em factor} of $(X, T)$. If in addition $\varphi$ is a homeomorphism, then it is called a ({\em topological}) {\em conjugacy (map)} and we say that $(X, T)$ and $(Y,S)$ are ({\em topologically}) {\em conjugate}. 

If $(X, T)$ is a topological dynamical system and $\rho$ is a compatible metric on $X$, then $\rho$ is necessarily complete since $X$ is compact. Let $(X, T)$ be a topological dynamical system and fix $\rho$ a compatible metric on $X$. We say that $(X, T)$ is {\em equicontinuous} if for all $\epsilon>0$ there is $\delta>0$ such that for all $n\in\mathbb{Z}$, if $\rho(x,y)<\delta$ then $\rho(T^nx, T^ny)<\epsilon$. Since $X$ is compact, the equicontinuity is a topological notion and does not depend on the compatible metric $\rho$.

Every topological dynamical system $(X, T)$ has a {\em maximal equicontinuous factor} (or {\em MEF}), i.e., an equicontinuous factor $(Y, S)$ with the factor map $\varphi$ such that if $(Z, G)$ is another equicontinuous factor of $(X, T)$ with factor map $\psi$ then there is a factor map $\theta: (Y,S)\to (Z, G)$ such that $\psi=\theta\circ \varphi$.

If $(X, T)$ is a topological dynamical system and $x\in X$, the {\em orbit} of $x$ is defined as $\{T^kx\,:\, k\in\mathbb{Z}\}$. If $A$ is a clopen subset of $X$, the {\em return times} of $x$ to $A$ is defined as $\Ret_A(x)=\{n\in\mathbb{Z}\,:\, T^nx\in A\}$. We regard $\Ret_A(x)$ as an element of $2^{\mathbb{Z}}=\{0,1\}^\mathbb{Z}$.

\subsection{Minimal Cantor systems\label{subsec:2.3}}
Recall that a {\em Cantor space} is a zero-dimensional, perfect, compact metrizable space. Let $X$ be a Cantor space and $T: X\to X$ be a homeomorphism. Then $(X, T)$ is called a {\em Cantor system}. $T$ is {\em minimal} if every orbit is dense, i.e., for all $x\in X$, $\{T^kx\,:\, k\in\mathbb{Z}\}$ is dense in $X$. A {\em minimal Cantor system} is a pair $(X, T)$ where $X$ is a Cantor space and $T: X\to X$ is a minimal homeomorphism.

Let $\mathcal{C}=2^{\mathbb{N}}=\{0,1\}^{\mathbb{N}}$ be the infinite product of the discrete space $\{0,1\}$ with the product topology. Then every Cantor space is homeomorphic to $\mathcal{C}$. Let $d_{\mathcal{C}}$ be the canonical compatible complete metric on $\mathcal{C}$, i.e., for $x, y\in \mathcal{C}$, if $x\neq y$ then
$$ d_{\mathcal{C}}(x,y)=2^{-n},\mbox{ where $n\in\mathbb{N}$ is the least such that $x(n)\neq y(n)$.}$$
Let
$$ \Aut(\mathcal{C})=\{T\,:\, \mbox{$T$ is a homeomorphism from $\mathcal{C}$ to $\mathcal{C}$}\} $$
be equipped with the compact-open topology, or equivalently the supnorm metric, i.e., for $T, S\in \Aut(\mathcal{C})$,
$$ d(T, S)=\sup\{d_{\mathcal{C}}(Tx, Sx)\,:\, x\in \mathcal{C}\}. $$
Then $\Aut(\mathcal{C})$ is a Polish space (c.f., e.g. \cite{KR}). Let $M(\mathcal{C})$ be the set of all minimal homeomorphisms of $\mathcal{C}$. Then for a $T\in \Aut(\mathcal{C})$, $T\in M(\mathcal{C})$ iff for all nonempty clopen $U\subseteq \mathcal{C}$, there is $N\in\mathbb{N}$ such that $\mathcal{C}=\bigcup_{-N\leq n\leq N} T^nU$.
This characterization implies that $M(\mathcal{C})$ is a $G_\delta$ subset of $\Aut(\mathcal{C})$, and hence $M(\mathcal{C})$ is also a Polish space. $M(\mathcal{C})$ is our coding space for all minimal Cantor systems.

We will also consider essentially minimal Cantor systems. A Cantor system $(X, T)$ is {\em essentially minimal} if it contains a unique minimal set, i.e., a nonempty closed $T$-invariant set which is minimal among all such sets.

\subsection{Ordered Bratteli diagrams} The concepts and terminology reviewed in this subsection are from \cite{HPS}, \cite{GPS} and \cite{DM}. Some notations are from \cite{Du} and \cite{DDMP}.
Recall that a {\em Bratteli diagram} is an infinite graph $(V,E)$ with the following properties:
\begin{itemize}
\item The vertex set $V$ is decomposed into pairwise disjoint nonempty finite sets $V=V_0\cup V_1\cup V_2\cup\cdots$, where $V_0$ is a singleton $\{v_0\}$;
\item The edge set $E$ is decomposed into pairwise disjoint nonempty finite sets $E=E_1\cup E_2\cup\cdots$;
\item For any $n\geq 1$, each $e\in E_n$ connects a vertex $u\in V_{n-1}$ with a vertex $v\in V_n$. In this case we write $\mathsf{s}(e)=u$ and $\mathsf{r}(e)=v$. Thus $\mathsf{s}, \mathsf{r}: E\to V$ are maps such that $\mathsf{s}(E_n)=V_{n-1}$ and $\mathsf{r}(E_n)=V_n$ for all $n\geq 1$.
\item $\mathsf{s}^{-1}(v)\neq\varnothing$ for all $v\in V$ and $\mathsf{r}^{-1}(v)\neq\varnothing$ for all $v\in V\setminus V_0$.
\end{itemize}

An {\em ordered Bratteli diagram} is a Bratteli diagram $(V,E)$ together with a partial ordering $\preceq$ on $E$ so that edges $e$ and $e'$ are $\preceq$-comparable if and only if $\mathsf{r}(e)=\mathsf{r}(e')$. 

A finite or infinite {\em path} in a Bratteli diagram $(V, E)$ is a sequence $(e_1, e_2, \dots)$ where  $\mathsf{r}(e_i)=\mathsf{s}(e_{i+1})$ for all $i\geq 1$. 
Given a Bratteli diagram $(V, E)$ and $0\leq n<m$, let $E_{n,m}$ be the set of all finite paths connecting vertices in $V_n$ and those in $V_m$. If $p=(e_{n+1},\dots, e_m)\in E_{n,m}$, define $\mathsf{r}(p)=\mathsf{r}(e_m)$ and $\mathsf{s}(p)=\mathsf{s}(e_{n+1})$. If in addition the Bratteli diagram is partially ordered by $\preceq$, then we also define a partial ordering $p\preceq' q$ for $p=(e_{n+1},\dots, e_m), q=(f_{n+1},\dots, f_m)\in E_{n,m}$ as either $p=q$ or there exists $n+1\leq i\leq m$ such that $e_i\neq f_i$, $e_i\preceq f_i$ and $e_j=f_j$ for all $i<j\leq m$. For an arbitrary strictly increasing sequence $(n_k)_{k\geq 0}$ of natural numbers with $n_0=0$, we define the {\em contraction} or {\em telescoping} of a Bratteli diagram $(V,E)$ with respect to $(n_k)_{k\geq 0}$ as $(V',E')$ where
$V'_k=V_{n_k}$ for $k\geq 0$ and $E'_k=E_{n_{k-1}, n_k}$ for $k\geq 1$. If in addition the given Bratteli diagram is ordered, then by contraction or telescoping we also obtain an ordered Bratteli diagram $(V',E',\preceq')$ with the order $\preceq'$ defined above. The inverse of the telescoping process is called {\em microscoping}. Two ordered Bratteli diagrams are {\em equivalent} if one can be obtained from the other by a sequence of telescoping and microscoping processes.

A Bratteli diagram $(V, E)$ is {\em simple} if there is a strictly increasing sequence $(n_k)_{k\geq 0}$ of natural numbers with $n_0=0$ such that the telescoping $(V', E')$ of $(V, E)$ with respect to $(n_k)_{k\geq 0}$ satisfies that for all $n\geq 1$, $u\in V'_{n-1}$ and $v\in V'_n$, there is $e\in E'_n$ with $\mathsf{s}(e)=u$ and $\mathsf{r}(e)=v$. This is equivalent to the property that for any $n\geq 1$ there is $m>n$ such that every pair of vertices $u\in V_n$ and $v\in V_m$ are connected by a finite path. It is obvious that if a Bratteli diagram $B$ is simple, then any Bratteli diagram equivalent to it is also simple. 

Given a Bratteli diagram $B=(V,E)$, define
$$ X_B=\{(e_n)_{n\geq 1}\,:\, e_n\in E_n, \mathsf{r}(e_n)=\mathsf{s}(e_{n+1}) \mbox{ for all $n\geq 1$}\}. $$
Since $X_B$ is a subspace of the product space $\prod_{n\geq 1}E_n$, we equip $X_B$ with the subspace topology of the product topology on $\prod_{n\geq 1}E_n$. An ordered Bratteli diagram $B=(V,E,\preceq)$ is {\em essentially simple} if there are unique elements $\xmax=(e_n)_{n\geq 1}, \xmin=(f_n)_{n\geq 1}\in X_B$ such that for every $n\geq 1$, $e_n$ is a $\preceq$-maximal element and $f_n$ is a $\preceq$-minimal element. $B=(V, E, \preceq)$ is {\em simple} if $(V, E)$ is simple and $B$ is essentially simple. If an ordered Bratteli diagram $B$ is (essentially) simple, then any ordered Bratteli diagram equivalent to it is also (essentially) simple.

Given an essentially simple ordered Bratteli diagram $B=(V,E,\preceq)$, we define the {\em Vershik map} $\lambda_B: X_B\to X_B$ as follows: $\lambda_B(\xmax)=\xmin$; if $(e_n)_{n\geq 1}\in X_B$ and $(e_n)_{n\geq 1}\neq \xmax$, then let 
$$ \lambda_B((e_1,e_2,\dots, e_k, e_{k+1},\dots))=(f_1,f_2,\dots, f_k, e_{k+1},\dots), $$
where $k$ is the least such that $e_k$ is not $\preceq$-maximal, $f_k$ is the $\preceq$-successor of $e_k$, and $(f_1,\dots, f_{k-1})$ is the unique path from $v_0$ to $\mathsf{s}(f_k)=\mathsf{r}(f_{k-1})$ such that   $f_i$ is $\preceq$-minimal for each $1\leq i\leq k-1$. Then $(X_B, \lambda_B)$ is an essentially minimal Cantor system (\cite{HPS}), which we call the {\em Bratteli--Vershik system} generated by $B$. If $B=(V, E,\preceq)$ is a simple ordered Bratteli diagram and $X_B$ is infinite, then $(X_B,\lambda_B)$ is a minimal Cantor system (\cite{GPS}). If two simple ordered Bratteli diagrams are equivalent, then the Bratteli--Vershik systems generated by them are conjugate, with the conjugacy map sending $\xmin$ to $\xmin$.

An essentially minimal Cantor system $(X, T)$ is of {\em finite topological rank} if it is conjugate to a Bratteli--Vershik system given by an essentially simple ordered Bratteli diagram $(V, E, \preceq)$ where $(|V_n|)_{n\geq 1}$ is bounded by a natural number $d$. The minimum possible value of $d$ is called the {\em topological rank} of the system, and is denoted by $\toprank(X,T)$. Such Bratteli diagrams have been called {\em finite rank Bratteli diagrams} in the literature, but \cite{Du} appears to be the first place where the terminology of finite topological rank was introduced.

An essentially minimal Cantor system $(X, T)$ with topological rank $1$ is called an {\em (infinite) odometer}. It is easy to see that any ordered Bratteli diagram for such an odometer is necessarily simple, and therefore an odometer is in fact minimal. The infinite odometers coincide with all equicontinuous minimal Cantor systems.

\subsection{Kakutani--Rohlin partitions\label{subsec:2.4}} The concepts and terminology reviewed in this subsection are again from \cite{HPS} and \cite{GPS}, with some notations from \cite{DDMP}.

For an essentially minimal Cantor system $(X, T)$, a {\em Kakutani--Rohlin partition} is a partition
$$ \mathcal{P}=\{ T^jB(k)\,:\, 1\leq k\leq d,\, 0\leq j<h(k)\} $$
of clopen sets, where $d, h(1), \dots, h(d)$ are positive integers and $B(1), \dots, B(d)$ are clopen subsets of $X$ such that 
$$ \bigcup_{k=1}^d T^{h(k)}B(k)=\bigcup_{k=1}^d B(k). $$
The set $B(\mathcal{P})=\bigcup_{k=1}^d B(k)$ is called the {\em base} of $\mathcal{P}$. For $1\leq k\leq d$, the subpartition $\mathcal{P}(k)=\{T^jB(k)\,:\, 0\leq j<h(k)\}$ is the {\em $k$-th tower} of $\mathcal{P}$, which has {\em base} $B(k)$ and {\em height} $h(k)$.

The following is a basic fact regarding the construction of Kakutani--Rohlin partitions.

\begin{lemma}[\cite{HPS}, Lemma~4.1]\label{lem:2.1} Let $(X, T)$ be an essentially minimal Cantor system, $Y$ be the unique minimal set,  $y\in Y$ and $Z$ be a clopen subset of $X$ containing $y$, and let $\mathcal{Q}$ be a finite partition of $X$ into clopen sets. Then there is a Kakutani--Rohlin partition $\mathcal{P}$ such that $y\in B(\mathcal{P})=Z$ and $\mathcal{P}$ refines $\mathcal{Q}$, i.e., every element of $\mathcal{Q}$ is a union of elements of $\mathcal{P}$.
\end{lemma}

The proof of the lemma gives a canonical construction of Kakutani--Rohlin partitions. Specifically, given $y\in Y$ and clopen set $Z$ containing $y$, the function $Z\to \mathbb{N}$, $x\mapsto n_x$, where $n_x$ is the least positive integer $n$ such that $T^nx\in Z$, is continuous. Thus by the compactness of $X$, $x\mapsto n_x$ is bounded. For any $h>0$, let $A_h=\{x\in Z\,:\, n_x=h\}$. Let $h(1), \dots, h(d)$ enumerate all $h>0$ where $A_h\neq\varnothing$. Then $\{T_jA_{h(k)}\,:\, 1\leq k\leq d,\, 0\leq j<h(k)\}$ is a Kakutani--Rohlin partition with base $Z$. 

Applying Lemma~\ref{lem:2.1} repeatedly, one quickly obtains the following theorem.

\begin{theorem}[\cite{HPS}, Theorem 4.2]\label{thm:1} For any essentially minimal Cantor system $(X, T)$ and $x$ in the unique minimal set, there exist
\begin{itemize}
\item positive integers $d_n$ for $n\geq 0$, with $d_0=1$,
\item positive integers $h_n(k)$ for $n\geq 0$ and $0\leq k<d_n$, with $h_0(1)=1$,
\item Kakutani--Rohlin partitions $\mathcal{P}_n$ for $n\geq 0$, where
$$ \mathcal{P}_n=\{T^jB_n(k)\,:\, 1\leq k\leq d_n,\, 0\leq j<h_n(k)\}, $$
with $B_0(1)=X$,
\end{itemize}
such that for all $n\geq 0$,
\begin{enumerate}
\item[(1)] each $\mathcal{P}_{n+1}$ refines $\mathcal{P}_n$, 
\item[(2)] $B(\mathcal{P}_{n+1})\subseteq B(\mathcal{P}_n)$,
\item[(3)] $\bigcap_n B(\mathcal{P}_n)=\{x\}$,
\item[(4)] $\bigcup_n\mathcal{P}_n$ generates the topology of $X$.
\end{enumerate}
\end{theorem}

We call the system of Kakutani--Rohlin partitions in Theorem~\ref{thm:1} a {\em nested system}. From such a system we define an ordered Bratteli diagram following \cite{HPS}. For each $n\geq 0$, let
$$ V_n=\{\mathcal{P}_n(k)\,:\, 1\leq k\leq d_n\}. $$
For $n\geq 1$, $1\leq k\leq d_n$, $1\leq \ell\leq d_{n-1}$ and $0\leq j<h_n(k)$, there is an edge $e_j\in E_n$ connecting $\mathcal{P}_n(k)$ to $\mathcal{P}_{n-1}(\ell)$ if $T^jB_n(k)\subseteq \bigcup_{0\le i< h_n(\ell)}T^iB_{n-1}(\ell)$. Then, if $e_{j_1}, \dots, e_{j_m}$ are all edges in $E_n$ connecting $\mathcal{P}_n(k)$ to some element of $V_{n-1}$, we set the partial ordering $\preceq$ by letting $e_j\preceq e_{j'}$ iff $j\leq j'$.
It was proved in \cite{HPS} that this ordered Bratteli diagram is essentially simple and that the Bratteli--Vershik system generated by this ordered Bratteli diagram is conjugate to $(X, T)$, with the conjugacy map sending $\xmin$ to $x$. If in addition $(X, T)$ is a minimal Cantor system, then the resulting ordered Bratteli diagram is necessarily simple.

Thus we have described a procedure to obtain an ordered Bratteli diagram given an essentially minimal Cantor system $(X, T)$ and a point $x$ in the unique minimal set. It was proved in \cite{HPS} that the equivalence class of the ordered Bratteli diagram does not depend on the choice of the Kakutani--Rohlin partitions in the procedure, i.e., all ordered Bratteli diagrams obtained through this procedure are equivalent.

Conversely, if $B=(V,E,\preceq)$ is an essentially simple ordered Bratteli diagram and $(X_B, \lambda_B)$ is the Bratteli--Vershik system generated by $B$, then there is a nested system of Kakutani--Rohlin partitions for $(X_B,\lambda_B)$ and $\xmin$ such that the ordered Bratteli diagram $B'$ defined above is equivalent to $B$. Thus, if an essentially minimal Cantor system $(X, T)$ has finite topological rank $d$, then there is a nested system of Kakutani--Rohlin partitions $\{\mathcal{P}_n\}_{n\geq 1}$ where $d_n=d$ for all $n\geq 1$.

\subsection{Subshifts}
The concepts and notation reviewed in this subsection are from \cite{GH} and \cite{GJJLLSW}. 

By a {\em finite word} we mean an element of $2^{<\mathbb{N}}=\{0,1\}^{<\mathbb{N}}=\bigcup_{N\in\mathbb{N}}\{0,1\}^N$. If $v$ is a finite word, we regard it as a function with domain $\{0,1,\dots, N-1\}$ for some $N\in\mathbb{N}$, and call $N$ its {\em length}, denoted as $|v|=N$. The {\em empty word} is the unique finite word with length $0$ (or the empty domain), and we denote it as $\varnothing$. If $v$ is a finite word and $s,t$ are integers such that $0\leq s\leq t\leq |v|-1$, then $v\rest[s,t]$ denotes the finite word $u$ of length $t-s+1$ where for $0\leq i<t-s+1$, $u(i)=v(s+i)$; $v\rest s$ denotes $v\rest[0,s]$, and is called a {\em prefix} or an {\em initial segment} of $v$. An {\em end segment} or a {\em suffix} of $v$ is $v\rest [s, |v|-1]$ for some $0\leq s\leq |v|-1]$. The empty word is both a prefix and a suffix of any word. Any word is also both a prefix and a suffix of itself. If $u, v$ are finite words, then $uv$ denotes the finite word $w$ of length $|u|+|v|$ where $w\rest [0,|u|-1]=u$ and $w\rest[|u|,|u|+|v|-1]=v$. For finite words $u, v$ with $|u|\leq |v|$, we say that $u$ is a {\em subword} of $v$ if there is $0\leq s\leq |v|-|u|$ such that $u=v\rest[s,s+|u|-1]$; when this happens we also say that $u$ {\em occurs} in $v$ at position $s$.

An {\em infinite word} is an element of $2^{\mathbb{N}}$, and a {\em bi-infinite word} is an element of $2^{\mathbb{Z}}$. For any infinite word $V\in 2^\mathbb{N}$ and integers $s, t$ with $0\leq s\leq t$, the notions $V\rest[s,t]$, $V\rest s$, finite subwords and their occurrences are similarly defined. For any bi-infinite word $x\in 2^\mathbb{Z}$ and integers $s, t$ with $s\leq t$, the notions $x\rest[s,t]$, finite subwords and their occurrences are also similarly defined.

We consider the {\em Bernoulli shift} on $2^{\mathbb{Z}}=\{0,1\}^{\mathbb{Z}}$, which is the homeomorphism $\sigma: 2^{\mathbb{Z}}\to 2^{\mathbb{Z}}$ defined by
$$ \sigma(x)(n)=x(n+1). $$
Since $2^{\mathbb{Z}}$ is homeomorphic to $\mathcal{C}=2^{\mathbb{N}}$, $(2^{\mathbb{Z}}, \sigma)$ is a Cantor system. A {\em subshift} $X$ is a closed $\sigma$-invariant subset of $2^{\mathbb{Z}}$. By a subshift we also refer to the Cantor system $(X, \sigma\rest X)$ or simply $(X,\sigma)$ when there is no danger of confusion. 

The following simple fact is a folklore.

\begin{lemma}\label{lem:2.4} An infinite subshift is not equicontinuous. In particular, it is not conjugate to any infinite odometer.
\end{lemma}


If $V\in 2^\mathbb{N}$ is an infinite word, let
$$ X_V=\{x\in 2^\mathbb{Z}\,:\, \mbox{ every finite subword of $x$ is a subword of $V$}\}. $$
Then $(X_V, \sigma)$ is a subshift and we call it the subshift {\em generated by} $V$. For any $V\in 2^\mathbb{N}$, $X_V$ is always nonempty. Note that for any $x\in X_V$ and finite subword $u$ of $x$, $u$ must occur in $V$ infinitely many times. We say that $V$ is {\em recurrent} if every finite subword of $V$ occurs in $V$ infinitely many times. When $V$ is recurrent, $X_V$ is either finite or a Cantor set, and $X_V$ is finite iff $V$ is {\em periodic}, i.e., there is a finite word $v$ such that $V=vvv\cdots$. Thus an infinite subshift generated by a recurrent $V$ is a Cantor system.

It is well known that all infinite odometers form a dense $G_\delta$ in the space $M(\mathcal{C})$ of all minimal Cantor systems. We give a proof of this fact in Corollary~\ref{cor:3.4} and Proposition~\ref{prop:odometerdense}.

\subsection{Subshifts of finite symbolic rank}
Some of the concepts and notation reviewed in this subsection are from \cite{GH} and \cite{GJJLLSW}, and some are new. 

Subshifts of finite symbolic rank are defined from infinite words of finite symbolic ranked constructions, whose definitions are inspired by the cutting-and-stacking processes that were used to construct measure-preserving transformations of finite rank (\cite{Fe}). We first define (symbolic) rank-1 subshifts, which are also called {\em Ferenczi subshifts} in \cite{AD} to honor the fact that Ferenczi popularized the concept in \cite{Fe}.

An infinite (symbolic) rank-1 word $V$ is defined as follows. Given a sequence of positive integers $\{r_n\}_{n\geq 0}$ with $r_n>1$ for all $n\geq 0$ (called the {\em cutting parameter}) and a doubly-indexed sequence of nonnegative integers $\{s_{n,i}\}_{n\geq 0, 1\leq i<r_n}$ (called the {\em spacer parameter}), a {\em (symbolic) rank-1 generating sequence} given by the parameters is the recursively defined sequence of finite words:
$$ \begin{array}{rcl} v_0&=& 0,\\ v_{n+1}&=&v_n1^{s_{n,1}}v_n\cdots v_n1^{s_{n,r_n-1}}v_n. \end{array} $$
Since $v_n$ is a prefix of $v_{n+1}$, it makes sense to define $V=\lim_n v_n$. This $V$ is called a {\em (symbolic) rank-1 word} and $X_V$ is called a {\em (symbolic) rank-1 subshift}. 

To generalize and define (symbolic) rank-$n$ subshifts we use the following concepts and notation. Let $\mathcal{F}$ be the set of all finite words in $2^{<\mathbb{N}}$ that begin and end with $0$. For a finite set $S\subseteq \mathcal{F}$ and finite word $w\in \mathcal{F}$, a {\em building} of $w$ from $S$ consists of a sequence $(v_1,\dots, v_{k+1})$ of elements of $S$ and a sequence $(s_1,\dots, s_k)$ of nonnegative integers for $k\geq 1$ such that
$$ w=v_11^{s_1}v_2\cdots v_k1^{s_k}v_{k+1}. $$
The sequence $(s_1,\dots, s_k)$ is called the {\em spacer parameter} of the building; it is {\em bounded by} $M$ if $s_1,\dots, s_k\leq M$. We say that {\em every word of $S$ is used} in this building if $\{v_1,\dots, v_{k+1}\}=S$. When there is a building of $w$ from $S$, we also say that $w$ is {\em built from} $S$; when the building consists of $(v_1,\dots, v_{k+1})$ and $(s_1,\dots, s_k)$, we also say that $w$ is built from $S$ {\em starting with} $v_1$. These notions can be similarly defined when the finite word $w$ is replaced by an infinite word $W$.

For $n\geq 1$, a {\em (symbolic) rank-$n$ generating sequence} is a doubly-indexed sequence $\{v_{i,j}\}_{i\geq 0, 1\leq j\leq n_i}$ of finite words satisfying, for all $i\geq 0$,
\begin{itemize}
\item $n_i\leq n$,
\item $v_{0,j}=0$ for all $1\leq j\leq n_0$,
\item $v_{i+1,1}$ is built from $S_i\triangleq\{v_{i,1},\dots, v_{i, n_i}\}$ starting with $v_{i,1}$,
\item $v_{i+1,j}$ is built from $S_i$ for all $2\leq j\leq n_{i+1}$.
\end{itemize}
A {\em (symbolic) rank-$n$ construction} is the (symbolic) rank-$n$ generating sequence $\{v_{i,j}\}_{i\geq 0, 1\leq j\leq n_i}$ together with exactly one building of $v_{i+1,j}$ from $S_i$ (for $v_{i+1,1}$, the building should start with $v_{i,1}$) for all $i\geq 0, 1\leq j\leq n_i$. We call $S_i$ the {\em $i$-th level} of the construction. The {\em spacer parameter} of the rank-$n$ construction is the collection of all spacer parameters of all the buildings in the construction; it is {\em bounded} if there is an $M>0$ such that all the spacer parameters of all the buildings in the construction are bounded by $M$. The (symbolic) rank-$n$ construction is {\em proper} if for all $i\geq 0$, $n_i=n$ and for all $1\leq j\leq n$, every word of $S_i$ is used in the building of each $v_{i+1,j}$. Since each $v_{i,1}$ is a prefix of $v_{i+1,1}$, it makes sense to define $V=\lim_{i}v_{i,1}$. 

Given a rank-$n$ construction with associated rank-$n$ generating sequence $\{v_{i,j}\}_{i\geq 0, 1\leq j\leq n_i}$, we define the set of all {\em expected subwords} of $v_{i,j}$, for $i\geq 0$ and $1\leq j\leq n_i$, inductively as follows: for each $v_{0,j}$, the set of all of its expected subwords is $\{v_{0,j}\}=\{0\}$; for $i\geq 0$, the set of all expected subwords of $v_{i+1,j}$ consists of
\begin{itemize}
\item $v_{i+1,j}$,
\item $u_1, \dots, u_{k+1}\in S_i$, where $(u_1,\dots, u_{k+1})$ and $(a_1,\dots, a_k)$ give the building of $v_{i+1,j}$ from $S_i$,
\item all expected subwords of $u_1,\dots, u_{k+1}\in S_i$.
\end{itemize}
Finally define the set of all {\em expected subwords} of $V=\lim_i v_{i,1}$ to be the union of the sets of all expected subwords of $v_{i,1}$ for all $i\geq 0$. Without loss of generality, we may assume that for all $i\geq 0$, all finite words in $S_i$ are expected subwords of $V$. It follows immediate from the construction that for all $i\geq 0$, the infinite word $V$ is built from $S_i$ starting with $v_{i,1}$.

Let $w\in \mathcal{F}$ and $S, T\subseteq\mathcal{F}$ are finite. Suppose $w$ is built from $S$ and that every word in $S$ is built from $T$. Then by {\em composing} the building of $w$ from $S$ with the buildings of each element of $S$ from $T$, we obtain a building of $w$ from $T$, and thus $w$ is also built from $T$. Given a rank-$n$ construction with associated rank-$n$ generating sequence $\{v_{i,j}\}_{i\geq 0, 1\leq j\leq n_i}$, and given $i<i'$, for all $1\leq j'\leq n_{i'}$ we obtain a building of $v_{i',j'}$ from $S_i$ by composing the buildings of elements of $S_{\iota}$ from $S_{\iota-1}$ for all $i+1\leq \iota\leq i'$. With this repeated composition process, we may obtain,  for any increasing sequence $\{i_k\}_{k\geq 0}$ with $i_0=0$, a rank-$n$ construction with associated rank-$n$ generating sequence $\{v_{i_k,j}\}_{k\geq 0, 1\leq j\leq n_{i_k}}$. Since $\lim_i v_{i,1}=\lim_{k}v_{i_k,1}$, the resulting infinite words are the same. We call this process {\em telescoping}. 

An infinite word $V$ is called a {\em (symbolic) rank-$n$ word} if it has a rank-$n$ construction but not a rank-$(n-1)$ construction. A subshift $X$ has {\em finite symbolic rank} if for some $n\geq 1$, $X=X_V$ where $V$ has a rank-$n$ construction; the smallest such $n$ is called the {\em symbolic rank} of $X$, and is denoted $\symbrank(X)=\symbrank(X, \sigma)$. 

By definition, if $\symbrank(X)=n$ then there is a rank-$n$ word $V$ such that $X=X_V$. 

Before closing this subsection, we give some examples. The best known example of a rank-1 generating sequence is the one coding the Chac\'{o}n transformation:
$$ v_0=0;\ v_{n+1}=v_nv_n1v_n. $$
The Morse sequence is the infinite 0,1-word generated by the Thue-Morse substitution $0\mapsto 01$ and $1\mapsto 10$. The subshift generated by the Morse sequence is an example of a minimal subshift of symbolic rank $2$. More generally, any $\mathcal{S}$-adic subshift of finite alphabet rank for which the initial alphabet $A_0=\{0,1\}$ (defined in Subsection~\ref{subsec:6.4}) is a natural example of subshifts of finite symbolic rank, where the symbolic rank is no more than the alphabet rank. Our main results in this paper will show that any minimal $\mathcal{S}$-adic subshift of finite  alphabet rank is conjugate to a minimal subshift of finite symbolic rank. 

Finally we turn to symbolic codings of interval exchange transformations, known as {\em IET subshifts}. When there are two intervals of irrational lengths, the interval exchange transformations are exactly irrational rotations and the corresponding IET subshifts are generated by Sturmian words. These subshifts have symbolic rank 2 (\cite{GJJLLSW}). In general, when there are more than two intervals and if the corresponding IET subshifts are minimal, they are known to have finite topological rank, and by our main results in this paper, they are conjugate to a minimal subshift of finite symbolic rank. If one considers the quotient systems of these subshifts where one of the intervals is coded by $0$ and the rest are coded by $1$, then the subshifts are natural examples of subshifts of finite symbolic rank.

\section{Some computations of descriptive complexity\label{sec:3}}

In this section we compute the descriptive complexity of various classes of Cantor systems. We first show that the class of all essentially minimal Cantor systems is a $G_\delta$ subset of $\Aut(\mathcal{C})$. Then we give a characterization of all essentially minimal Cantor systems with bounded topological rank. As a consequence, we show that for each $n\geq 1$, the class of all essentially minimal Cantor systems of topological rank $\leq n$ is a $G_\delta$ subset of $\Aut(\mathcal{C})$. This implies that for each $n\geq 1$, the class of all minimal Cantor systems of topological rank $\leq n$ is a $G_\delta$ subset of $M(\mathcal{C})$. 

We first give a characterization of essential minimality for a Cantor system $(X, T)$. We say a subset $A$ of $X$ has the {\em finite covering property} if there is some $N\in\mathbb{N}$ such that $\bigcup_{-N\leq n\leq N}T^nA=X$.

\begin{proposition}\label{thm:3.1} Let $(X, T)$ be a Cantor system and let $\rho\leq 1$ be a compatible metric on $X$. Then the following are equivalent:
\begin{enumerate}
\item $(X, T)$ is essentially minimal.
\item For any clopen set $A$ of $X$, if $A$ has the finite covering property then there is a clopen subset $B$ of $A$ with the finite covering property such that $\diam(B)\leq \diam(A)/2$.
\end{enumerate}
\end{proposition}

\begin{proof} First assume $(X, T)$ is essentially minimal. Suppose $A$ is a clopen subset of $X$ with the finite covering property, i.e. for some $N\in\mathbb{N}$ we have $\bigcup_{-N\leq n\leq N}T^nA=X$. Let $x$ be an arbitrary element of the unique minimal set of $X$. Then for some $-N\leq n\leq N$, $x\in T^nA$, where $T^nA$ is still clopen. Let $Y$ be a clopen subset of $T^nA$ containing $x$ such that $\diam(T^{-n}Y)\leq \diam(A)/2$. By Theorem 1.1 of \cite{HPS}, $\bigcup_{k\in\mathbb{Z}}T^kY=X$. By the compactness of $X$, $Y$ has the finite covering property. Let $B=T^{-n}Y$. Then $B\subseteq A$, $\diam(B)\leq \diam(A)/2$ and $B$ also has the finite covering property. 

Conversely, assume (2) holds. Starting with $A_0=X$ and repeatedly applying (2), we obtain a decreasing sequence $\{A_n\}_{n\geq 0}$ of clopen subsets of $X$ such that $\diam(A_n)\leq 2^{-n}$ and each $A_n$ has the finite covering property. Let $x$ be the unique element of $\bigcap_n A_n$. Then any clopen subset $B$ of $X$ containing $x$ has the finite covering property. By Theorem 1.1 of \cite{HPS}, $(X, T)$ is essentially minimal.
\end{proof}

Let $E(\mathcal{C})$ be the set of all essentially minimal homeomorphisms of $\mathcal{C}$. 

\begin{corollary} $E(\mathcal{C})$ is a $G_\delta$ subset of $\Aut(\mathcal{C})$, hence is a Polish space.
\end{corollary}

\begin{proof} Note that for any clopen subset $A$ of $\mathcal{C}$, $A$ has the finite covering property for $(\mathcal{C}, T)$ is an open condition for $T\in\Aut(\mathcal{C})$. Thus condition (2) of Proposition~\ref{thm:3.1} gives a $G_\delta$ condition for $T\in \Aut(\mathcal{C})$.
\end{proof}

We next give a characterization of essentially minimal Cantor systems of bounded topological rank.

\begin{theorem}\label{thm:3.3} Let $(X, T)$ be an essentially minimal Cantor system, $\rho\leq 1$ be a compatible complete metric on $X$, and $n\geq 1$. The following are equivalent:
\begin{enumerate}
\item[(1)] $(X, T)$ has topological rank $\leq n$.
\item[(2)] There exists $x\in X$ such that for all $\epsilon>0$, there is a Kakutani--Rohlin partition $\mathcal{P}$ with no more than $n$ many towers such that $\diam(A)<\epsilon$ for all $A\in\mathcal{P}$, $\diam(B(\mathcal{P}))<\epsilon$, and $x\in B(\mathcal{P})$.
\item[(3)] For any clopen subset $Z$ of $X$ with the finite covering property, for any finite partition $\mathcal{Q}$ into clopen sets, there is a Kakutani--Rohlin partition $\mathcal{P}$ with no more than $n$ many towers such that $B(\mathcal{P})\subseteq Z$, $\diam(B(\mathcal{P}))\leq \diam(Z)/2$ and $\mathcal{P}$ refines $\mathcal{Q}$.
\end{enumerate}
\end{theorem}

\begin{proof} We first show (1)$\Rightarrow$(2). Suppose $(X, T)$ has topological rank $\leq n$. Then there is an essentially simple ordered Bratteli diagram $B=(V,E,\preceq)$ such that $(X, T)$ is conjugate to the Bratteli--Vershik system $(X_B,\lambda_B)$ generated by $B$, and for all $k\geq 1$, $|V_k|\leq n$. It suffices to verify that (2) holds for $(X_B,\lambda_B)$. Let $y=\xmin$. Each level $V_k$ of $B$ gives rise to a Kakutani--Rohlin partition $\mathcal{P}_k$, where each set in $\mathcal{P}_k$ corresponds to a path from $V_0$ to a vertex in $V_k$, and $B(\mathcal{P}_k)$ consists of all the basic open sets correspondent to the minimal paths from $V_0$ to each vertex of $V_k$. Since $y$ is the unique minimal infinite path, we have $\bigcap_k B(\mathcal{P}_k)=\{y\}$. Let $\eta\leq 1$ be the standard metric on $X_B$. Let $\epsilon>0$. Then there is a large enough $k$ such that all the minimal paths from $V_0$ to the vertices of $V_k$ agree on the first $k'<k$ many edges where $2^{-k'}<\epsilon$. This implies that $\diam(B(\mathcal{P}_k))\leq 2^{-k'}<\epsilon$. Also, for all $A\in \mathcal{P}_k$, $\diam(A)\leq 2^{-k}<2^{-k'}<\epsilon$. $\mathcal{P}_k$ has $|V_k|\leq n$ many towers. Since $y\in B(\mathcal{P}_k)$, we have that $\mathcal{P}_k$ witnesses (2) for $(X_B, \lambda_B)$.

Next we show (2)$\Rightarrow$(3). Suppose (2) holds for $x\in X$. We note that for any $n\in \mathbb{Z}$, (2) holds also for $T^nx$. This is because, the property described in (2) is invariant under topological conjugacy, and $T^n:(X, T)\to (X, T)$ is a topological conjugacy sending $x$ to $T^nx$. Let $Z$ be a clopen subset of $X$ with the finite covering property. Without loss of generality assume $\dim(Z)>0$. Then $Z$ meets every orbit in $X$, and therefore there is $x\in Z$ such that the property in (2) holds. Let $\mathcal{Q}$ be a finite partition of $X$ into clopen sets. Let 
$\delta>0$ be the infimum of $d(y,z)$ where $y, z$ are from different elements of $\mathcal{Q}$.
Let $\xi>0$ be such that $x\in \{y\in X\,:\, \rho(x,y)<\xi\}\subseteq Z$ and $\diam(Z)>2\xi$.
Let $\epsilon=\min\{\delta, \xi\}>0$. Let $\mathcal{P}$ be a Kakutani--Rohlin partition with no more than $n$ many towers such that $\diam(A)<\epsilon$ for all $A\in\mathcal{P}$, $\diam(B(\mathcal{P}))<\epsilon$, and $x\in B(\mathcal{P})$. Then $B(\mathcal{P})\subseteq \{y\in X\,:\, \rho(x,y)<\xi\}\subseteq Z$, $\diam(B(\mathcal{P}))<\xi<\diam(Z)/2$, and $\mathcal{P}$ refines $\mathcal{Q}$ because for any $A\in\mathcal{P}$, $\diam(A)<\delta$.

Finally we prove (3)$\Rightarrow$(1). Assume $(X, T)$ is essentially minimal and (3) holds. Note that the base of any Kakutani--Rohlin partition has the finite covering property. By applying (3) repeatedly, we obtain a system of Kakutani--Rohlin partitions $\{\mathcal{P}_k\}_{k\geq 0}$ so that $\mathcal{P}_0=\{X\}$, each $\mathcal{P}_{k+1}$ refines $\mathcal{P}_k$, $B(\mathcal{P}_{k+1})\subseteq B(\mathcal{P}_k)$, $\diam(B(\mathcal{P}_{k+1}))\leq \diam(B(\mathcal{P}_k))/2$, each $\mathcal{P}_k$ consists of no more than $n$ many towers, and $\bigcup_k\mathcal{P}_k$ generates the topology of $X$. Let $x$ be the unique element of $\bigcap_kB(\mathcal{P}_k)$. Then any clopen subset of $X$ containing $x$ has the finite covering property. By Theorem 1.1 of \cite{HPS}, $x$ is in the unique minimal set of $X$. Now $\{\mathcal{P}_k\}_{k\geq 0}$ is a nested system of Kakutani--Rohlin partitions in the sense of Theorem~\ref{thm:1}, which gives rise to an ordered Bratteli diagram for $(X, T)$ with each level consisting of no more than $n$ vertices. Thus $\toprank(X, T)\leq n$.
\end{proof}

\begin{corollary}\label{cor:3.4} For any $n\geq 1$, the set of all essentially minimal $T\in \Aut(\mathcal{C})$ with topological rank $\leq n$ is a $G_\delta$ subset of $E(\mathcal{C})$. Similarly, for any $n\geq 1$, the set of all minimal $T\in \Aut(\mathcal{C})$ with topological rank $\leq n$ is a $G_\delta$ subset of $M(\mathcal{C})$.
\end{corollary}

\begin{proof} This follows immediately from clause (3) of Theorem~\ref{thm:3.3}.
\end{proof}

We also have the following immediate corollary regarding the descriptive complexity of (essentially) minimal Cantor systems with finite topological rank.

\begin{corollary}\label{cor:3.45} The set of all essentially minimal $T\in\Aut(\mathcal{C})$ with finite topological rank is a ${\bf\Sigma}^0_3$ subset of $E(\mathcal{C})$. Similarly, the set of all minimal $T\in\Aut(\mathcal{C})$ with finite topological rank is a ${\bf\Sigma}^0_3$ subset of $M(\mathcal{C})$.
\end{corollary}

Here we remark that the proof of Theorem~\ref{thm:3.3} implies that in clauses (2) and (3) of Theorem~\ref{thm:3.3} one may replace ``no more than $n$ many towers" with ``exactly $n$ many towers." Similar proofs also give that, if the system $(X, T)$ considered is minimal, then in clause (2) of Theorem~\ref{thm:3.3} one may replace ``there exists $x\in X$" with either ``for nonmeager many $x\in X$" or ``for comeager many $x\in X$."

Finally we note that the set of all infinite odometers form a dense $G_\delta$ in $M(\mathcal{C})$.

\begin{proposition}\label{prop:odometerdense} The set of all infinite odometers is a dense $G_\delta$ in the space of all minimal Cantor systems.
\end{proposition}

\begin{proof} 
Since the set of all infinite odometers is just the set of all minimal Cantor systems of topological rank $1$, it is a $G_\delta$ in the space of all minimal Cantor systems by Corollary~\ref{cor:3.4}. We only verify that it is dense. Let $(X, T)$ be a minimal Cantor system and suppose $\mathcal{P}$ is a clopen partition of $X$. We only need to define an infinite odometer $S$ on $X$ such that $SZ=TZ$ for all $Z\in \mathcal{P}$. Consider $\tilde{T}=T^{-1}$. Then $(X,\tilde{T})$ is again a mininal Cantor system. If we define an infinite odometer $\tilde{S}$ on $X$ such that $\tilde{S}^{-1}Z=\tilde{T}^{-1}Z$ for all $Z\in\mathcal{P}$, then $S=\tilde{S}^{-1}$ is again an infinite odometer, and $SZ=TZ$ holds for all $Z\in\mathcal{P}$. Thus we focus on $(X, \tilde{T})$ in the rest of this proof.

By Lemma~\ref{lem:2.1} $\mathcal{P}$ can be refined by a Kakutani--Rohlin partition for $\tilde{T}$. Therefore, without loss of generality, we may assume that $\mathcal{P}$ itself is a Kakutani--Rohlin partition. Suppose $\mathcal{P}=\{ \tilde{T}^jB(k)\,:\, 1\leq k\leq d,\, 0\leq j<h(k)\}$. 

We define a directed graph $G=(V,E)$, where $V=\{v_1,\cdots,v_d\}$ has $d$ vertices, and for any $1\leq k,k'\leq d$, there is a directed edge $e\in E$ from $v_k$ to $v_{k'}$ iff $B(k')\cap \tilde{T}^{h(k)}B(k)\ne\varnothing$. It follows from the minimality of $\tilde{T}$ that $G=(V,E)$ is strongly connected, i.e, there is a directed path from any vertex to any other vertex. Now fix a finite sequence $p=(e_1,\cdots,e_m)$ of edges in $G$ such that $(e_1,\cdots,e_m,e_1)$ is a directed path and $\{e_1,\cdots,e_m\}=E$. Then $p$ is a directed cycle in $G$. 

Consider an edge $e\in E$, say $e$ is from $v_k$ to $v_{k'}$. Let $n_e$ be the number of times $e$ appears in $p$. Let $A_e=B(k)\cap \tilde{T}^{-h(k)}(B(k'))$. Then $A_e$ is a clopen set in $X$. Let $\{A_{e,1},\dots, A_{e, n_e}\}$ be a partition of $A_e$ into $n_e$ many clopen subsets of $X$. If $e$ appears in $p$ as $e_{i_1},\dots, e_{i_{n_e}}$, we associate with each $e_{i_j}$ the set $A_{e,j}$ for $1\leq j\leq n_e$.

Thus we have obtained disjoint nonempty clopen sets $C_1,\cdots,C_m$ such that $\mathcal{Q}\triangleq\{C_1,\cdots,C_m\}$ is a partition of $B(\mathcal{P})$, and for any $1\le i\le m$, if $e_i$ is an edge from $v_k$ to $v_{k'}$ then $C_i\subseteq B(k)$ and $\tilde{T}^{h(k)}C_i\subseteq B(k')$. We define an odometer $\tilde{S}\,:\, X\to X$ such that for any $1\le i\le m$, if $e_i$ is an edge from $v_k$ to $v_{k'}$ then $\tilde{S}^{h(k)}C_i=C_{i+1}$ (with $C_{m+1}=C_1$) and $\tilde{S}^{j}C_i=\tilde{T}^{j}C_i$ for $1\le j<h(k)$. In fact $\mathcal{Q}$ is a Kakutani--Rohlin partition for $\tilde{S}$ (with one tower), and $\tilde{S}$ is defined by recursive refinements starting with $\mathcal{Q}$. It is now clear that $\tilde{S}^{-1}Z=\tilde{T}^{-1}Z$ for all $Z\in \mathcal{P}$ as desired.
\end{proof}

\section{A characterization of minimal rank-$1$ subshifts\label{sec:5}}

In this section we give an explicit topological characterization for all minimal Cantor systems which are conjugate to infinite rank-$1$ subshifts. In contrast to the results in Section~\ref{sec:3}, the descriptive complexity of this characterization will be on a higher level than $G_\delta$. 

Define 
$$\mathcal{Z}=\left\{x\in 2^\mathbb{Z}\,:\, \forall n\ \exists m>n\ x(m)=0\mbox{ and } \forall n\ \exists m>n\ x(-m)=0\right\}. $$
Then $\mathcal{Z}$ is a $\sigma$-invariant dense $G_\delta$ subset of $2^\mathbb{Z}$.

For a bi-infinite word $x\in \mathcal{Z}$ and a finite word $v\in\mathcal{F}$, we say that $x$ is {\em built from} $v$ if $\sigma^n(x)$ can be written in the form
$$ \sigma^n(x)=\cdots v1^{s_{-2}}v1^{s_{-1}}v1^{s_0} \cdot  v1^{s_1}v1^{s_2}\cdots $$
for an bi-infinite sequence $(\cdots, s_{-2}, s_{-1}, s_0, s_1, s_2, \cdots)$ of nonnegative integers and for some $n\in\Z$. For finite words $u,v\in\mathcal{F}$, we say that $u$ is {\em built from} $v$ if there are nonnegative integers $s_1,\dots, s_k$ for $k\geq 1$ such that
$$ u=v1^{s_1}v\cdots v1^{s_k}v. $$
The demonstrated occurrences of $v$ in $u$ are called {\em expected occurrences}.

\begin{lemma}\label{lem:5.1} Let $x\in \mathcal{Z}$ and $v\in\mathcal{F}$. Then the following are equivalent:
\begin{enumerate}
\item[(i)] $x$ is built from $v$.
\item[(ii)] For all $m\in\mathbb{N}$ there exists a finite word $u$ such that $x\rest[-m,m]$ is a subword of $u$ and $u$ is built from $v$.
\end{enumerate}
\end{lemma}

\begin{proof} The implication (i)$\Rightarrow$(ii) is immediate. We show (ii)$\Rightarrow$(i). Let $n$ be the number of $0$s in $v$, i.e., $n$ is the number of distinct occurrences of $0$ in $v$. Let $m_0$ be large enough such that $x\rest[-m_0,m_0]$ contains at least $n$ many $0$s. Let $k_1<\dots<k_n\in [-m_0,m_0]$ be such that $x(k_i)=0$ for all $1\leq i\leq n$ and that if $k_1\leq k\leq k_n$ is such that $x(k)=0$ then $k=k_i$ for some $1\leq i\leq n$. By (ii), for each $m\geq m_0$ there is a finite word $u$ such that $x\rest[-m,m]$ is a subword of $u$ and $u$ is built from $v$. Exactly one of $k_1, \dots, k_n$ corresponds to a starting position of an expected occurrence of $v$ in $u$. We denote this value of $k\in\{k_1,\dots, k_n\}$ as $k(m)$. Let $k_\infty\in \{k_1,\dots, k_n\}$ be such that for infinitely many $m\geq m_0$, $k(m)=k_\infty$. Let $M_\infty$ be the infinite set such that for all $m\in M_\infty$, $k(m)=k_\infty$. Then $v$ occurs in $x$ starting at position $k_\infty$. We claim that for all $k>k_\infty$ such that $x(k)=0$ and there are a multiple of $n$ many $0$s from $k_\infty$ to $k-1$, $v$ occurs in $x$ starting at position $k$. This is because, fixing such a $k$ and letting $m\in M_\infty$ with $m\geq k+|v|$, there is a finite word $u$ such that $x\rest[-m,m]$ is a subword of $u$ and $u$ is built from $v$; since the occurrence of $v$ starting at position $k_\infty$ corresponds to an expected occurrence of $v$ in $u$, it follows that there is another expected occurrence of $v$ in $u$ starting at the position corresponding to $k$, and so $v$ occurs in $x$ starting at position $k$. By a similar argument we can also prove a claim that for all $k<k_\infty$ such that $x(k)=0$ and there are a multiple of $n$ many $0$s from $k$ to $k_\infty-1$, $v$ occurs in $x$ starting at position $k$. Putting these two claims together, we conclude that $x$ is built from $v$.
\end{proof}

Lemma~\ref{lem:5.1} implies immediately that for any $v\in\mathcal{F}$, the set of all $x\in \mathcal{Z}$ such that $x$ is built from $v$ is closed in $\mathcal{Z}$. 

Let $(X, T)$ be a Cantor system and let $A$ be a clopen subset of $X$. Define $\mathcal{B}_T(A)$ to be the smallest Boolean algebra $\mathcal{B}$ of subsets of $X$ such that $T^nA\in \mathcal{B}$ for all $n\in\mathbb{Z}$. We say that $(T, A)$ is {\em generating} if $\mathcal{B}_T(A)$ contains all clopen subsets of $X$.

\begin{theorem}\label{thm:5.3} Let $(X, T)$ be a minimal Cantor system and $x_0\in X$. Then the following are equivalent:
\begin{enumerate}
\item[(1)] $(X, T)$ is conjugate to a (infinite) rank-$1$ subshift.
\item[(2)] There is a clopen subset $A$ of $X$ such that $(T,A)$ is generating and for all $n\in\mathbb{N}$ there is a $v\in\mathcal{F}$ satisfying:
\begin{itemize}
\item $|v|\geq n$ and $\Ret_A(x_0)$ is built from $v$, and
\item for any $u\in\mathcal{F}$ such that $|u|\geq |v|$ and $\Ret_A(x_0)$ is built from $u$, there exists $u'\in \mathcal{F}$ such that $|u'|\le |u|+|v|$, $u'$ is built from $v$, and $u$ is an initial segment of $u'$.
\end{itemize}
\end{enumerate}
\end{theorem}

\begin{proof} Clause (2) is apparently conjugacy invariant, thus to see (1)$\Rightarrow$(2), we may assume $V$ is a rank-$1$ word, $X=X_V$ is an infinite minimal rank-$1$ subshift, and $T=\sigma$. Let $A=\{x\in X\,:\, x(0)=1\}$. Then $(T, A)$ is generating, and $\Ret_A(x_0)=x_0$. The set of all finite words $v$ such that $V$ is built from $v$ is a subset of the set of all finite words $v$ such that $x_0$ is built from $v$. Now given any $n\in\mathbb{N}$, let $v\in\mathcal{F}$ be such that $V$ is built {\em fundamentally} from $v$ (see Definition~2.13 of \cite{GH}). Then by Proposition~2.16 of \cite{GH}, for any $u\in\mathcal{F}$ such that $|u|\geq |v|$ and $V$ is built from $u$, $u$ is built from $v$. This proves (2) by Proposition~2.36 of \cite{GH}.

Conversely, assume $A$ is a clopen subset of $X$ witnessing (2). Since $(T,A)$ is generating, the map $\Ret_A: X\to 2^\mathbb{Z}$ is a homeomorphic embedding such that $\Ret_A\circ T=\sigma\circ \Ret_A$. Thus $\Ret_A(X)$ is a minimal subshift, and $\Ret_A$ is a conjugacy map. By repeatedly applying (2), we obtain an infinite sequence of finite words $\{v_n\}_{n\geq 0}$ in $\mathcal{F}$ such that $\Ret_A(x_0)$ is built from each $v_n$ and for all $n\geq 0$, $v_n$ is an initial segment of $v_{n+1}$ and $v_{n+1}$ is an initial segment of some $u$ which is built from $v_n$. This allows us to define an infinite word $V=\lim_n v_n$. By definition, $V$ is a rank-$1$ word. To finish the proof it suffices to verify that $\Ret_A(X)=X_V$. By the minimality of $\Ret_A(X)$, for any $y\in \Ret_A(X)$, the set of all finite subwords of $y$ coincides with the set of all finite subwords of $\Ret_A(x_0)$. On the other hand, our assumption guarantees that the set of all finite subwords of $\Ret_A(x_0)$ coincides with the set of all finite subwords of $V$. Thus $\Ret_A(X)=X_V$ and $X$ is conjugate to $X_V$, a rank-$1$ subshift.
\end{proof}

The apparent descriptive complexity given by clause (2) of the above theorem is ${\bf\Sigma}^0_5$, which is significantly more complex than $G_\delta$.

\section{Proper finite rank constructions\label{sec:4}}
The following is a basic property regarding symbolic rank-$n$ constructions.

\begin{proposition}\label{prop:4.1} Let $n\geq 1$. Suppose $\{T_i\}_{i\geq 0}$ is a sequence of finite subsets of $\mathcal{F}$ such that $T_0=\{0\}$ and for all $i\geq 0$, $|T_i|\leq n$ and each element of $T_{i+1}$ is built from $T_i$. Then there is a rank-$n$ construction with associated rank-$n$ generating sequence $\{v_{i,j}\}_{i\geq 0, 1\leq j\leq n_i}$ such that for all $i\geq 0$, $v_{i,1},\dots, v_{i, n_i}\in T_i$.
\end{proposition}

\begin{proof} For each $i\geq 0$ and $v\in T_{i+1}$, fix a building of $v$ from $T_i$. Define a binary relation $R$ on $\bigcup_{i\geq 0}T_i$ by $R(u,v)$ if for some $i\geq 0$, $u\in T_i$, $v\in T_{i+1}$, and the building of $v$ from $T_i$ starts with $u$. Let $<$ be the transitive closure of $R$. Then $<$ is a (strict) partial order on $\bigcup_{i\geq 0}T_i$. We inductively define an infinite $R$-chain of words $\{u_i\}_{i\geq 0}$, i.e., $u_i\in T_i$ and $R(u_i, u_{i+1})$ for all $i\geq 0$. Let $u_0=0$. Note that there are infinitely many words $u\in \bigcup_{i\geq 0}T_i$ such that $u_0<u$ (in fact $u_0<u$ for all $u\in \bigcup_{i\geq 0}T_i$). In general, assume $u_i$ has been defined such that there are infinitely many $w\in \bigcup_{i\geq 0}T_i$ with $u_i<w$. In particular the set $W=\{w\in \bigcup_{j\geq i+2}T_j\,:\, u_i<w\}$ is infinite. Note that for each $w\in W$ there is a $u_w\in T_{i+1}$ such that $R(u_i,u_w)$ and $u_w<w$. Since $T_{i+1}$ is finite, there is a $v\in T_{i+1}$ such that for infinitely many $w\in W$, $u_w=v$. Let $u_{i+1}=v$. Then there are infinitely many $w\in \bigcup_{i\geq 0}T_i$ such that $u_{i+1}<w$. This finishes the inductive construction. 

Now define $v_{i,j}$ for each $i\geq 0$ so that $v_{i,1}=u_i$ and $\{v_{i,1}, \dots, v_{i,n_i}\}=T_i$, where $n_i=|T_i|$. With the fixed buildings, this gives a rank-$n$ construction as required.
\end{proof}

Next we characterize the rank-$n$ subshifts which have proper rank-$n$ constructions. We use $1^\mathbb{Z}$ to denote the element $x\in 2^\mathbb{Z}$ where $x(k)=1$ for all $k\in\mathbb{Z}$.

\begin{theorem}\label{thm:4.2} Let $n\geq 1$ and let $X$ be a subshift of symbolic rank $n$. The following are equivalent:
\begin{enumerate}
\item[(1)] There exists a rank-$n$ word $V$ such that $X=X_V$, and $V$ has a proper rank-$n$ construction.
\item[(2)] For any rank-$n$ word $V$ such that $X=X_V$, $V$ has a proper rank-$n$ construction.
\item[(3)] For any $x\in X$ such that $x\neq 1^\mathbb{Z}$, the orbit of $x$ is dense in $X$.
\end{enumerate}
\end{theorem}

\begin{proof} We first show (1)$\Rightarrow$(3). Suppose $V$ is a rank-$n$ word such that $X=X_V$, and $V$ has a proper rank-$n$ construction with associated rank-$n$ generating sequence $\{v_{i,j}\}_{i\geq 0,1\le j\le n}$. For each $i\geq 0$, define $a_i=\max_{1\le j\le n}|v_{i,j}|$. Let  $x\in X_V$ and assume $x\ne 1^\mathbb{Z}$. There exists an $m\in \mathbb{Z}$ so that $x(m)=0$. Fix an $i\geq 0$ and consider the finite word $u=x\rest [m-a_{i+1}, m+a_{i+1}]$. By the definition of $X_V$, $u$ is a subword of $V$.  Since $V$ is built from $S_{i+1}$, by considering the length of $u$ we get that there is $1\le j_0\le n$ such that $v_{i+1,j_0}$ is a subword of $u$. By the properness of the rank-$n$ construction, $v_{i,1}$ is a subword of $v_{i+1,j_0}$, and hence a subword of $x$. This implies that the orbit of $x$ is dense in $X_V$.

Next we show (3)$\Rightarrow$(2). Let $V$ be a rank-$n$ word such that $X=X_V$. Suppose for any $x\in X_V$ such that $x\ne 1^\mathbb{Z}$, the orbit of $x$ is dense in $X_V$. We fix a rank-$n$ construction with associated rank-$n$ generating sequence $\{v_{i,j}\}_{i\geq 0,1\le j\le n_i}$, where $V=\lim_i v_{i,1}$. Since $V$ is a rank-$n$ word, it does not have a rank-$(n-1)$ construction; by telescoping if necessary, we may assume that $n_i=n$ for all $i\geq 0$. Also, without loss of generality, we assume that for all $i\geq 0$, all finite words in $S_i$ are expected subwords of $V$. In particular, if $i<i'$, then every finite word in $S_i$ is an expected subword of some word in $S_{i'}$.

Next we claim that for any $i_0>0$ and $1\leq j_0\leq n$, there exists $i>i_0$ such that for any $1\le j\le n$, $v_{i_0,j_0}$ is an expected subword of $v_{i,j}$. Assume not; then for any $i>i_0$ there is $1\le j\le n$ such that $v_{i_0,j_0}$ is not an expected subword of $v_{i,j}$. We define a sequence $\{T_k\}_{k\geq 0}$ of finite subsets of $\mathcal{F}$ as follows. Let $T_0=\{0\}$. For $k>0$, let $T_k$ be the set of all $v_{i_0+k,j}$ for $1\leq j\leq n$ such that $v_{i_0,j_0}$ is not an expected subword of $v_{i_0+k,j}$. Then for all $k\geq 0$, $T_k\subseteq S_{i_0+k}$ and so $|T_k|\leq n-1$. Also, for all $k\geq 0$, each element of $T_{k+1}$ is built from $T_k$. By Proposition~\ref{prop:4.1}, there is a rank-$(n-1)$ construction with associated rank-$(n-1)$ generating sequence $\{w_{k,\ell}\}_{k\geq 0, 1\leq\ell\leq m_k}$ such that for all $k\geq 0$ and $1\leq \ell\leq m_k$, $w_{k,\ell}\in T_k\subseteq S_{i_0+k}$. Let $W=\lim_k w_{k,1}$. Then every finite subword of $W$ is a subword of $V$. Hence $X_W$ is an closed invariant subset of $X_V$. It is clear from the construction of $W$ that there is $x\in X_W$ such that $x\neq 1^\mathbb{Z}$. Since the orbit of $x$ is dense in $X_V$, we get that $X_W=X_V$. This contradicts our assumption that $\symbrank(X_V)=n$.

Using the claim, and by telescoping, we obtain a proper rank-$n$ construction for $V$.

Finally, (2)$\Rightarrow$(1) is immediate.
\end{proof}

Note that the implications (1)$\Rightarrow$(3) in the above theorem do not require that $X$ is of symbolic rank $n$.

\begin{corollary}\label{cor:4.2} Let $n\geq 1$ and $X$ be an infinite subshift of symbolic rank $\leq n$. Suppose $X=X_V$ and $V$ has a proper rank-$n$ construction. Then $(X, \sigma)$ is an essentially minimal Cantor system. In particular, there is $k\in\mathbb{N}$ such that $0^k$ is not a subword of $V$.
\end{corollary}

\begin{proof} Since $X=X_V$ where $V$ has a proper rank-$n$ construction, $V$ is recurrent. Since $X_V$ is infinite, it is a Cantor set. Now if $1^\mathbb{Z}\not\in X$, then by Theorem~\ref{thm:4.2} (3) $X$ is minimal; if $1^\mathbb{Z}\in X$ then $\{1^\mathbb{Z}\}$ is invariant and by Theorem~\ref{thm:4.2} (3) it is the unique minimal set in $X$. Thus $(X, \sigma)$ is an essentially minimal Cantor system. In either case, $0^\mathbb{Z}\not\in X_V$, thus there is $k\in\mathbb{N}$ such that $0^k$ is not a subword of $V$.
\end{proof}

Note that any rank-$1$ construction is proper, and thus any infinite rank-$1$ subshift is an essentially minimal Cantor system.

\begin{corollary}\label{cor:4.3} Let $n\geq 1$ and let $X$ be an infinite subshift of symbolic rank $n$. Then the following are equivalent:
\begin{enumerate}
\item[(1)] $X$ is minimal.
\item[(2)] There exists a rank-$n$ word $V$ such that $X=X_V$, and $V$ has a proper rank-$n$ construction with bounded spacer parameter.
\item[(3)] For any rank-$n$ word $V$ such that $X=X_V$, $V$ has a proper rank-$n$ construction with bounded spacer parameter.
\end{enumerate}
\end{corollary}

\begin{proof} To see (1)$\Rightarrow$(3), suppose $X$ is minimal. Then $1^{\mathbb{Z}}\not\in X$ and clause (3) of Theorem~\ref{thm:4.2} holds. By Theorem~\ref{thm:4.2}, for any rank-$n$ word $V$ such that $X=X_V$, $V$ has a proper rank-$n$ construction with associate rank-$n$ generating sequence $\{v_{i,j}\}_{i\geq 0, 1\leq j\leq n}$. Without loss of generality, we may assume that every word in this sequence is an expected subword of $V$. We claim that this given proper rank-$n$ construction has bounded spacer parameter. Otherwise there are arbitrarily large $k$ with $1^k$ as a subword of $V$, and then $1^{\mathbb{Z}}\in X_V=X$, a contradiction.

The implication (3)$\Rightarrow$(2) is immediate.

Finally, we prove (2)$\Rightarrow$(1). Suppose $V$ is a rank-$n$ word such that $X=X_V$, and $V$ has a proper rank-$n$ construction with bounded spacer parameter. Then $1^\mathbb{Z}\not\in X_V$, and by Theorem~\ref{thm:4.2} $X$ is minimal.
\end{proof}

Again, we remark that the implications (2)$\Rightarrow$(1) of the above corollary do not require that $X$ be a subshift of symbolic rank $n$.

\section{Finite symbolic rank and finite topological rank\label{sec:6}}
In this section we prove that minimal subshifts of finite symbolic rank have finite topological rank, and conversely, any minimal Cantor system of finite topological rank is either an odometer or conjugate to a subshift of finite symbolic rank. Together with previous results (\cite{DM} and \cite{DDMP}), our results show that the three notions of finite rank for minimal expansive Cantor systems all coincide with each other. 

\subsection{From finite symbolic rank to finite topological rank}

We first consider minimal subshifts of finite symbolic rank.

The following concept of unique readability will be useful in our proofs to follow. Let $n\geq 1$. Fix a symbolic rank-$n$ construction with associated rank-$n$ generating sequence $\{v_{i,j}\}_{i\geq 0, 1\leq j\leq n}$. Let $V=\lim_i v_{i,1}$. Without loss of generality assume every $v_{i,j}$ is an expected subword of $V$, and that for each $i\geq 1$, the words $v_{i,1}, \dots, v_{i,n}$ are distinct. For $x\in X_V$, a {\em reading} of $x$ is a sequence $\{E_i\}_{i\geq 0}$ satisfying, for each $i\geq 0$,
\begin{enumerate}
\item[(i)] each element of $E_i$ is a pair $(k,j)$, where $1\leq j\leq n$ and $k$ is the starting position of an occurrence of $v_{i,j}$ in $x$;
\item[(ii)] if $(k_1,j_1), (k_2, j_2)\in E_i$ and $k_1<k_2$, then $k_1+|v_{i,j_1}|\leq k_2$;
\item[(iii)] $E_0=\{(k,j)\,:\, x(k)=0 \mbox{ and } j=1\}$; and
\item[(iv)] for each $(k,j)\in E_{i}$, there is exactly one $(k',j')\in E_{i+1}$ such that $k'\le k$ and $k'+|v_{i+1,j'}|\ge k+|v_{i,j}|$. 
\end{enumerate}
If every $x\in X_V$ has a unique reading, we say that $\{v_{i,j}\}_{i\geq 0,1\le j\le m}$ has {\em unique readability}, and we call an occurrence (starting at position) $k$ of $v_{i,j}$ in $x$  {\em expected} if $(k,j)\in E_i$ for the unique reading of $x$. Every rank-$1$ generating sequence whose induced infinite rank-$1$ word is not periodic has unique readability (Proposition 2.29 of \cite{GH}).

\begin{lemma}\label{lem:6.1}  Let $n\geq 1$, $\{v_{i,j}\}_{i\geq 0,1\le j\le n}$ be a rank-$n$ generating sequence, and $V=\lim_i v_{i,1}$. Then any $x\in X_V$ has a reading.
\end{lemma}

\begin{proof} We fix a rank-$n$ construction of $V$ with associated rank-$n$ generating sequence $\{v_{i,j}\}_{i\geq 0, 1\leq j\leq n}$. Without loss of generality we may assume that every $v_{i,j}$ is an expected subword of $V$, and that for each $i\geq 1$, the words $v_{i,1},\dots, v_{i,n}$ are distinct. For each $i\geq 0$, define $a_i=\max_{1\le j\le n}|v_{i,j}|$ and $b_i=\inf_{1\leq j\leq n}|v_{i,j}|$. Then $b_{i+1}\geq 2b_i$ for all $i\geq 0$.

We consider several cases. Case 1: $x=1^\mathbb{Z}$. In this case a unique reading is given by $E_i=\varnothing$ for all $i\geq 0$. 

Case 2: There exists $k_0\in\mathbb{Z}$ such that $x(k_0)=0$ and $x(k)=1$ for all $k<k_0$. First fix any $i\geq 0$. Define $u_i=x\rest[k_0-a_i, k_0+a_i]$. Since $u_i$ is a subword of $V$, $V$ is built from $\{v_{i,1}, \dots, v_{i,n}\}$, and $a_i\geq |v_{i,j}|$ for all $1\leq j\leq n$, we have that for some $1\leq j_0\leq n$, $k_0$ is the starting position of an occurrence of $v_{i,j_0}$ in $x$. Now following the rank-$n$ construction of $V$, by an induction on $t=i, i-1,\dots, 0$, we define collections $E^i_{t}$ for $t\le i$ as follows. First let $E^i_{i}=\{(k_0,j_0)\}$. Suppose now $E^i_{t}$ has been defined, which is a collection of some pairs $(k,j)$, where $k$ is the starting position of an occurrence of $v_{t,j}$ in $x$. Assume that for $(k_1, j_1), (k_2,j_2)\in E^i_t$ with $k_1<k_2$, we have $k_1+|v_{t,j_1}|\leq k_2$. Now for each $(k,j)\in E^i_t$, the building of $v_{t,j}$ from $\{v_{t-1,1},\dots, v_{t-1,n}\}$ in the fixed rank-$n$ construction gives rise to pairs $(k',j')$, where $v_{t-1,j'}$ occurs at position $k'$ and this occurrence corresponds to the occurrence of $v_{t-1,j'}$ in the building of $v_{t,j}$ as an expected subword. We put all such $(k',j')$ in $E^i_{t-1}$. It is clear that for $(k_1',j_1'), (k_2',j_2')\in E^i_{t-1}$ with $k_1'<k_2'$, we have $k_1'+|v_{t-1,j_1'}|\leq k_2'$. It is also clear that for each $(k',j')\in E^i_{t-1}$, there is exactly one $(k,j)\in E^i_{t}$ such that $k\le k'$ and $k+|v_{t,j}|\ge k'+|v_{t-1,j'}|$. Finally, we note that $E^i_0=\{(k,j)\,:\, x(k)=0, k_0\le k\le k_0+|v_{i,j}|-1, \mbox{ and } j=1\}$. This finishes the definition of $E^i_t$ for $t\leq i$. We have that for all $k_0\leq k\leq k_0+b_i+1$, $(k,1)\in E^i_0$ iff $x(k)=0$.

For $i\geq 0$, define $e_i\in \{0,1\}^{\mathbb{N}\times\mathbb{Z}\times\{1,\cdots,n\}}$ by letting $e_i(t, k,j)=1$ iff $t\le i$ and $(k,j)\in E^i_{t}$. Since $\{0,1\}^{\mathbb{N}\times\mathbb{Z}\times\{1,\cdots,n\}}$ is compact, there exists an accumulation point $e$ of $\{e_i\}_{i\geq 0}$. For each $t\geq 0$, define $E_t=\{(k,j)\,:\, e(t,k,j)=1\}$. Since $\{b_i\}_{i\geq 0}$ is strictly increasing, we conclude that for all $k\geq k_0$, $(k,1)\in E_0$ iff $x(k)=0$. The other properties of a reading are also easily verified. Thus $\{E_t\}_{t\geq 0}$ is a reading of $x$.

Case 3: There exists $k_0\in\mathbb{Z}$ such that $x(k_0)=0$ and $x(k)=1$ for all $k>k_0$. This case is similar to Case 2.

Case 4: For any $k\in\mathbb{Z}$ there are $k_1<k<k_2$ such that $x(k_1)=x(k_2)=0$.
Let $k_0$ be an integer satisfying $x(k_0)=0$. For $i\geq 0$, let $\ell_{i,1}$ be the $(2a_i+1)$th natural number such that $x(k_0+\ell_{i,1})=0$; let $\ell_{i,2}$ be the $(2a_i+1)$th natural number such that $x(k_0-\ell_{i,2})=0$. Define $u_i=x\rest[k_0-\ell_{i,2}, k_0+\ell_{i,1}]$. Then $u_i$ is a subword of $V$. Since $V$ is built from $\{v_{i,1},\dots, v_{i,n}\}$, by the definition of $a_i$, there exist $m_i<k_0$ and a subword $w_i$ of $V$ such that 
\begin{itemize}
\item $w_i$ is of the form $v_{i,j_1}1^{s_1}v_{i,j_2}1^{s_2}v_{i,j_3}$, where $1\leq j_1, j_2, j_3\leq n$ and $s_1, s_2\geq 0$,
\item $m_i$ is the starting position of an occurrence of $w_i$ in $x$, and
\item $m_i+|v_{i,j_1}1^{s_1}|\le k_0\le m_i+|v_{i,j_1}1^{s_1}v_{i,j_2}|-1$. 
\end{itemize}
Now we proceed as in the proof of Case 2 to define $E^i_t$ for all $t\leq i$ and finally obtain a reading $\{E_t\}_{t\geq 0}$ of $x$ by compactness. 
\end{proof}

Next we define a concept that guarantees unique readability. Let $n\geq 1$. We say a rank-$n$ construction with associated rank-$n$ generating sequence $\{v_{i,j}\}_{i\geq 0, 1\leq j\leq n}$ is {\em good} if it is proper and for any $i\geq 0$ and $1\le j\le n$, $v_{i,j}$ is not of the form $$\alpha 1^{s_1}v_{i,j_1}1^{s_2}v_{i,j_2}\cdots v_{i,j_{k-1}}1^{s_k}\beta $$ 
where $k\geq 1$, $\alpha$ is a nonempty suffix of some $v_{i,j_{k}}$, and $\beta$ is a nonempty prefix of some $v_{i,j_{k+1}}$. If a rank-$n$ construction is good, we say that the infinite word $V=\lim_i v_{i,1}$ is {\em good}. 

\begin{lemma} Consider a good rank-$n$ construction with associate rank-$n$ generating sequence $\{v_{i,j}\}_{i\geq 0, 1\leq j\leq n}$. Then $\{v_{i,j}\}_{i\geq 0, 1\leq j\leq n}$ has unique readability.
\end{lemma}

\begin{proof} Let $x\in X_V$. Let $\{E_i\}_{i\geq 0}$ be a reading of $x$, which exists by Lemma~\ref{lem:6.1}. By the definition of a reading, we have that for each $i\geq 0$, $E_i$ gives a way in which $x$ is built from $\{v_{i,1}, \dots, v_{i,n}\}$. Now suppose $\{E_i'\}_{i\geq 0}$ is another reading of $x$, and suppose $i\geq 0$ is the smallest such that $E_i\neq E'_i$. Without loss of generality, let $(k,j)\in E_i\setminus E'_i$. 

Consider two cases. Case 1: there is $j'\neq j$ such that $(k,j')\in E'_i$. In this case without loss of generality assume $|v_{i,j}|>|v_{i,j'}|$.
Then $v_{i,j}$ can be written in the form $v_{i,j'}1^{s_1}v_{i,j_1}\cdots v_{i,j_\ell}1^{s_{\ell+1}}\beta$ for some $\ell\geq 0$ and nonempty $\beta$ where $\beta$ is a prefix of some $v_{i,j_{\ell+1}}$. This contradicts the assumption that our rank-$n$ construction is good. 

Case 2: there is no $j'$ such that $(k,j')\in E'_i$. By the definition of a reading, there is a unique $(k',j')\in E'_i$ where $k'<k$ such that $k\leq k'+|v_{i,j'}|$. If $k'+|v_{i,j'}|\leq k+|v_{i,j}|$ then $v_{i,j}$ can be written in the form $\alpha 1^{s_1}v_{i,j_1}\cdots v_{i,j_{\ell}}1^{s_{\ell+1}}\beta$, contradicting the assumption that our rank-$n$ construction is good. If $k'+|v_{i,j'}|>k+|v_{i,j}|$ then $v_{i,j'}$ can be written in the form $\alpha 1^{s_1}v_{i,j_1}\cdots v_{i,j_{\ell}}1^{s_{\ell+1}}\beta$, again contradicting our assumption.
\end{proof}

Note that the definition of goodness does not rule out the possibility that some $v_{i,j}$ is a subword of $v_{i,j'}$ for $j'\neq j$.

\begin{lemma}\label{lem:6.3} Suppose $V$ has a good rank-$n$ construction with associated rank-$n$ generating sequence $\{v_{i,j}\}_{i\geq 0, 1\leq j\leq n}$. The for any $i\geq 0$, $1\leq j\leq n$ and $k\in\mathbb{Z}$, the set
$$ \{x\in X_V\,:\, \mbox{ there is an expected occurrence of $v_{i,j}$ at position $k$}\} $$
is clopen in $X_V$.
\end{lemma}

\begin{proof} Let $E_{i,j,k}$ denote the set in question. Then $x\in E_{i,j,k}$ iff for any $0\leq t\leq |v_{i,j}|-1$, $x(k+t)=v_{i,j}(t)$ and for any $1\leq j'\leq n$ and $k'\leq k$, if $j'\neq j$ and $k'+|v_{i,j'}|\geq k+|v_{i,j}|$, then there is $0\leq s\leq |v_{i,j'}|-1$ such that $x(k'+s)\neq  v_{i,j'}(s)$. This implies that $E_{i,j,k}$ is clopen.
\end{proof}

\begin{proposition}\label{prop:6.3} Let $n\geq 1$ and $X$ be an infinite subshift of symbolic rank $\leq n$. Suppose $X=X_V$ and $V$ has a proper rank-$n$ construction. Then there exists a word $W$ with a good rank-$2n$ construction such that $X$ is a factor of $X_W$. Moreover, if in addition $X$ is minimal, then $W$ can be chosen so that $X_W$ is minimal.
\end{proposition}

\begin{proof}
Fix a proper rank-$n$ construction of $V$ with associated rank-$n$ generating sequence $\{v_{i,j}\}_{i\geq 0,1\le j\le n}$. Without loss of generality assume $v_{i,1},\dots, v_{i,n}$ are distinct for all $i\geq 1$. By Corollary~\ref{cor:4.2} there is a $k_0\in\mathbb{N}$ such that $0^{k_0}$ is not a subword of $V$; we fix such a $k_0$. For all $i\geq 0$, define $a_i=\max_{1\leq j\leq n}|v_{i,j}|$ and $b_i=\inf_{1\leq j\leq  n}|v_{i,j}|$. Then $b_{i+1}\geq nb_i$ for all $i\geq 0$, hence in particular $b_i\geq n^i$ for all $i\geq 0$.

We define a rank-$2n$ generating sequence $\{w_{p,q}\}_{p\geq 0, 1\leq q\leq 2n}$. Let $w_{0,q}=0$ for  $1\leq q\leq 2n$. To define $w_{1,q}$, let $i_1\geq 0$ be such that $b_{i_1}>4k_0+4$. Then define 
$$ w_{1,q}=\left\{\begin{array}{ll} v_{i_1,q} & \mbox{ if $1\leq q\leq n$,} \\ \\
 010^{|v_{i_1,q-n}|-4}10 &\mbox{ if $n+1\leq q\leq 2n$.}\end{array}
\right.
$$
Note that for all $1\leq j\leq n$, $|w_{1,j}|=|w_{1,j+n}|=|v_{i_1,j}|$.

For $p\geq 1$, suppose $i_p$ has been defined and $w_{p,q}$ have been defined for all $1\leq q\leq 2n$. We define $i_{p+1}$ and $w_{p+1,q}$ as follows. First set 
$$ m_p=\left\lceil\frac{a_{i_p+1}}{b_{i_p}}\right\rceil. $$ 
Then let $i_{p+1}>i_p$ be large enough such that by telescoping using the buildings in the proper rank-$n$ construction $\{v_{i,j}\}_{i\geq 0, 1\leq j\leq n}$, we can write, for all $1\le j\le n$, $v_{i_{p+1},j}$ in the form 
\begin{equation}\label{eqn:1} v_{i_p,j_1}1^{s_1}\cdots 1^{s_\ell}v_{i_p,j_{\ell+1}} \end{equation}
with $\ell>12m_p+4n$. This is doable since $\ell\geq n^{i_{p+1}-i_p}$. Note that for $j=1$ we have $j_1=1$. We also note the following property (*) of the word in (\ref{eqn:1}): 
for any $1\le t\le \ell+2-2m_p$, $\{j_t,\cdots,j_{t+2m_p-1}\}=\{1,\cdots,n\}$.
This is because, 
\begin{equation}\label{eqn:3} u\triangleq v_{i_p, j_t}1^{s_t}\cdots v_{i_p,j_{t+2m_p-1}} \end{equation}
consists of $2m_p$ many consecutive expected occurrences of subwords of $v_{i_{p+1},j}$ of the form $v_{i_p, j'}$; since $v_{i_{p+1}, j}$ is built from $\{v_{i_p+1,1},\dots, v_{i_p+1,n}\}$, by our definition of $m_p$, $u$ must contain some expected occurrence of $v_{i_p+1,j'}$, where $1\leq j'\leq n$, as a subword. Hence (*) holds by the properness of the construction.

We now fix $1\leq j\leq n$ and assume that $v_{i_{p+1},j}$ is in the form (\ref{eqn:1}). For $1\leq t\leq \ell+1$, define
$$ \phi(t)=\left\{\begin{array}{ll} j_t+n & \mbox{ if $2\leq t\leq 2m_p+j+1$ or $\ell-2m_p-j+1\leq t\leq \ell$,} \\ j_t & \mbox{ otherwise}
\end{array}\right. 
$$
and
$$ \psi(t)=\left\{\begin{array}{ll}j_t+n & \mbox{ if $2\leq t\leq 2m_p+j+n+1$ or $\ell-2m_p-j-n+1\leq t\leq \ell$,} \\ j_t & \mbox{ otherwise.}
\end{array}\right.
$$
Then define
$$ w_{p+1,j}= w_{p,\phi(1)}1^{s_1}\cdots 1^{s_\ell}w_{p,\phi(\ell+1)} $$
and
$$ w_{p+1,j+n}=w_{p,\psi(1)}1^{s_1}\cdots 1^{s_\ell}w_{p, \psi(\ell+1)}. $$
This finishes the defintion of $\{w_{p,q}\}_{p\geq 0, 1\leq q\leq 2n}$. 

We verify that the construction defined is proper, i.e., for all $p\geq 1$ and $1\leq q\leq 2n$, all words in $\{w_{p,1},\dots, w_{p,2n}\}$ are used in the building of $w_{p+1,q}$. We first assume $1\leq q\leq n$. Since $\ell>12m_p+4n$, there exists $1\leq t_0\leq \ell+1$ such that for all $t_0\leq t\leq t_0+2m_p-1$, $\phi(t)=j_t$. By property (*), $\{\phi(t_0),\dots, \phi(t_0+2m_p-1)\}=\{j_{t_0},\dots, j_{t+2m_p-1}\}=\{1,\dots, n\}$. Thus all words in $\{w_{p,1}, \dots, w_{p,n}\}$ are used in the building of $w_{p+1,q}$. On the other hand, for $2\leq t\leq 2m_p+1$, $\phi(t)=j_t+n$. By property (*) again, $\{\phi(2),\dots, \phi(2m_p+1)\}=\{j_2+n,\dots, j_{2m_p+1}+n\}=\{n+1,\dots, 2n\}$. Hence all words in $\{w_{p,n+1},\dots, w_{p,2n}\}$ are also used in the building of $w_{p+1,q}$. The case $n+1\leq q\leq 2n$ is similar.

Next we claim that 
\begin{itemize}
\item for all $p\geq 2$ and $1\leq q\leq 2n$, $w_{p,q}$ is not of the form
\begin{equation}\label{eqn:2} \alpha 1^{r_1}w_{p,q_1}1^{r_2}w_{p,q_2}\cdots w_{p,q_{d-1}}1^{r_d}\beta
\end{equation}
where $d\geq 1$, $\alpha$ is a nonempty suffix of some $w_{p,q_{d}}$ and $\beta$ is a nonempty prefix of some $w_{p,q_{d+1}}$, and 
\item for all $p\geq 2$ and $1\leq q, q'\leq 2n$, if $q\neq q'$ then $w_{p,q}$ is not a subword of $w_{p,q'}$. 
\end{itemize}
We prove this claim by induction on $p\geq 2$.

First suppose $p=2$. We observe that $w_{2,q}$ can be written as $u_1yu_2$, where $0^{k_0}$ is not a subword of $y$, $y$ begins and ends with 0, every word of $\{v_{i_1,1},\dots, v_{i_1,n}\}$ occurs at least 3 different times in $y$, and both $u_1$ and $u_2$ are of the form
\begin{equation}\label{eqn:4} \alpha 01^{s_1}010^{h_1}101^{s_2}010^{h_2}10\cdots 010^{h_{2m_p+q}}101^{s_{2m_p+q+1}}0\beta \end{equation}
where $\alpha,\beta$ have lengths at least $3k_0$, $0^{k_0}$ is not a subword of either $\alpha$ or $\beta$, and $h_t>4k_0$ for all $1\leq t\leq 2m_p+q$. The statement about $y$ is based on the observation that, by (\ref{eqn:1}), $y$ can be taken to contain subwords of the form (\ref{eqn:3}) for three different values $2m_p+2n+1< t_1<t_2<t_3< \ell-4m_p-2n$ where $t_3-t_2, t_2-t_1> 2m_p$; by property (*), for each value $t_1, t_2, t_3$, the subword of the form (\ref{eqn:3}) contains a distinct occurrence of each word in $\{v_{i_1,1},\dots, v_{i_1,n}\}$. Note that for $q'\neq q$, $w_{2,q'}$ does not have a subword of the form (\ref{eqn:4}), hence $w_{2,q}$ is not a subword of $w_{2,q'}$. Now suppose $w_{2,q}$ can be written in the form of (\ref{eqn:2}), then by the above observation, there are $1\leq q_1, q_2\leq 2n$, a nonempty suffix $y_1$ of $w_{2,q_1}$, and a nonempty prefix $y_2$ of $w_{2,q_2}$, such that $w_{2,q}=y_11^{s}y_2$ for some nonnegative integer $s$. First suppose $q_1=q_2=q$. Then $w_{2,q}$ must have a subword of the form $z\triangleq 0^{2k_0}101^{s_1}v_{i_1,j_1}1^{s}v_{i_1,j_2}1^{s_2}010^{2k_0}$, and in fact $y$ must be a subword of $z$. However, note that $y$ has at least $3$ different occurrences of each word in $\{v_{i_1,1},\dots, v_{i_1,n}\}$, while $z$ does not have this property, a contradiction. Next suppose $q_1\ne q$. Then $u_1$ is not a subword of $w_{2,q_1}$, so $y_1$ is a prefix of $u_1$. It follows that $yu_2$ is a suffix of $y_2$, $q_2=q$, and $y_2$ must be $w_{2,q}$ itself, contradicting the assumption that $y_1$ is nonempty. The case $q_2\ne q$ is similar. This completes the proof of the claim for $p=2$.

Suppose the claim holds for $p\geq 2$. We verify it for $p+1$. First we observe that for any $1\leq q\leq 2n$, $w_{p+1,q}$ can be written as $u_1yu_2$, where $u_1$ and $u_2$ are of the form
\begin{equation}\label{eqn:5} w_{p,q_1}1^{s_1}w_{p,q_2}1^{s_2}\cdots w_{p, q_{2m_p+q+1}}1^{s_{2m_p+q+1}}w_{p,q_{2m_p+q+2}} \end{equation}
where $1\leq q_1, q_{2m_p+q+2}\leq n$, $n+1\leq q_t\leq 2n$ for all $2\leq t\leq 2m_p+q+1$, and by inductive hypothesis, if $w_{p,\kappa}$ is a subword of $y$, then $1\leq \kappa\leq n$. By the inductive hypothesis, if $q\neq q'$ then $w_{p+1,q'}$ does not contain a subword of the form (\ref{eqn:5}), hence $w_{p+1,q}$ is not a subword of $w_{p+1,q'}$. Next assume $w_{p+1,q}$ can be written in the form (\ref{eqn:2}) with $p+1$ replacing $p$. Then by the above observation, there are $1\leq q_1, q_2\leq 2n$, a nonempty suffix $y_1$ of $w_{p+1,q_1}$, and a nonempty prefix $y_2$ of $w_{p+1, q_2}$ such that $w_{p+1,q}=y_11^sy_2$ for some nonnegative integer $s$. First suppose $q_1=q_2=q$. Then $w_{p+1,q}$ has a subword of the form $z\triangleq w_{p, j_1}1^{s_1}w_{p,j_2}1^{s}w_{p, j_3}1^{s_2}w_{p, j_4}$, where $n+1\leq j_1, j_4\leq 2n$ and $1\leq j_2, j_3\leq n$. In fact, $y$ must be a subword of $z$. However, $y$ contains at least $3$ different occurrences of words in $\{w_{p, 1}, \dots, w_{p,n}\}$, a contradiction. Next suppose $q_1\ne q$. then $u_1$ is not a subword of $w_{p+1,q_1}$, so $y_1$ is a prefix of $u_1$. It follows that $yu_2$ is a suffix of $y_2$, $q_2=q$, and $y_2$ must be $w_{p+1,q}$ itself, contradicting the assumption that $y_1$ is nonempty. The case $q_2\ne q$ is similar. This completes the proof of the claim. 

In view of the claim, if we define $\{w'_{p,q}\}_{p\geq 0, 1\leq q\leq 2n}$ by letting $w'_{0,q}=0$ and $w'_{p,q}=w_{p+1,q}$ for $p\geq 1$ and $1\leq q\leq 2n$, then we obtain a good proper rank-$2n$ construction. Let $W=\lim_p w'_{p,1}$. Note that $W=\lim_p w_{p, 1}$. 

We define a factor map $\varphi: X_W\to X_V$. For $x\in X_W$ and $k\in\mathbb{Z}$, if there is $1\leq j\leq n$ such that the position $k$ is part of an expected occurrence of $w'_{1,j}$ or $w'_{1,j+n}$ which starts at position $k'\le k$, then let $\varphi(x)(k)=v_{i_2,j}(k-k')$; otherwise let $\varphi(x)(k)=1$. By the unique readability, and since for all $1\leq j\leq n$, $|w'_{1,j}|=|w'_{1,j+n}|=|w_{2,j}|=|w_{2,j+n}|=|v_{i_2,j}|$, $\varphi$ is well defined. By Lemma~\ref{lem:6.3} $\varphi$ is continuous. It is clear that $\varphi$ is a factor map. 

Finally, if $X_V$ is minimal, then the construction associated with $\{v_{i,j}\}_{i\geq 0, 1\leq j\leq n}$ must have bounded spacer parameter, because otherwise $1^\mathbb{Z}\in X_V$ and $\{1^\mathbb{Z}\}$ is invariant. Now it follows from our construction that the defined proper rank-$2n$ construction of $W$ also has bounded spacer parameter, and by the implication (2)$\Rightarrow$(1) of Corollary~\ref{cor:4.3} (which does not require the assumption on the symbolic rank of $X_W$), $X_W$ is minimal.
\end{proof}

The following is a corollary to the proof of Proposition~\ref{prop:6.3}.

\begin{proposition}\label{prop:6.4} Let $n\geq 1$ and $X$ be an infinite subshift of symbolic rank $\leq n$. Suppose $X=X_V$ and $V$ has a proper rank-$n$ construction $\{v_{i,j}\}_{i\geq 0, 1\leq j\leq n}$ which has unique readability. Then for any $i\geq 0$, $1\leq j\leq n$, and $k\in\mathbb{Z}$, the set 
$$\{x\in X\,:\,\mbox{ there is an expected occurrence of $v_{i,j}$ in $x$ at position $k$}\}$$ is clopen in $X$.
\end{proposition}

\begin{proof} Let $W$ be the infinite word with a good rank-$2n$ construction with associated rank-$2n$ generating sequence $\{w'_{p,q}\}_{p\geq 0, 1\leq q\leq 2n}$ and let $\varphi: X_W\to X_V$ be the factor map both given in the proof of Proposition~\ref{prop:6.3}. Given $i\geq 0$, $1\leq j\leq n$, and $k\in\mathbb{Z}$, let $E_{i,j,k}$ denote the set in question.
It suffices to show that $E_{i,j,k}$ is clopen for all $i=i_p$ for some $p>1$.

Suppose $i=i_p$ for some $p>1$. Note that a reading of $y\in X_W$ can determine a reading of $\varphi(y)$. Thus $\varphi^{-1}(E_{i,j,k})$ consists exactly of those $y\in X_W$ such that there is an expected occurrence of $w'_{p,j}$ or $w'_{p,j+n}$ in $y$ at position $k$. By our construction, $\varphi^{-1}(E_{i,j,k})$ is easily seen to be clopen. Thus $E_{i,j,k}$ is clopen.
\end{proof}

We are now ready to compute the topological rank of a minimal Cantor system if it has a good construction.

\begin{proposition}\label{prop:6.5} Let $n\geq 1$. Let $X$ be an infinite minimal subshift of symbolic rank $\leq n$. Suppose $X=X_V$ and $V$ has a good rank-$n$ construction. Then $X$ has finite topological rank. 
\end{proposition}

\begin{proof}
Fix a good rank-$n$ construction with associated rank-$n$ generating sequence $\{v_{i,j}\}_{i\geq 0,1\le j\le n}$. Let $\ell\in\mathbb{N}$ be such that $1^\ell$ is not a subword of $V$. Note that for each $i\geq 0$, there are at least 4 distinct expected occurrences of $v_{i,1}$ in $v_{i+3,1}$. We let $k_i$ be the starting position of the second expected occurrence of $v_{i,1}$ in $v_{i+3,1}$. Let $x_0$ be the unique element of $2^\mathbb{Z}$ such that for all $i\geq 0$, there exists an occurrence of $v_{3i,1}$ which starts at the position $-\sum_{0\le i'\le i-1} k_{3i'}$. Then every finite subword of $x_0$ is a subword of $v_{3i,1}$ for some $i\geq 0$, and thus $x_0\in X_V$. 

Now for every $m\geq 2$, let $A_m$ be the set of all $x\in X_V$ such that there is an expected occurrence of $v_{3m,1}$ in $x$ starting at the position $-\sum_{0\le i\le m-1}k_{3i}$, which is the second expected occurrence of $v_{3m,1}$ in an expected occurrence of $v_{3m+3,j}$ in $x$ for some $1\leq j\leq n$. By Lemma~\ref{lem:6.3} each $A_m$ is clopen in $X_V$. By definition $x_0\in A_m$.

Now consider the canonical Kakutani-Rohlin partition $\mathcal{P}$ with base $A_m$ defined in the remark after Lemma~\ref{lem:2.1}. The number of towers in $\mathcal{P}$ corresponds to the number of different $h>0$ such that $h$ is the smallest positive integer with $\sigma^h(x)\in A_m$ for some $x\in A_m$. Suppose $x\in A_m$ and let $1\leq j\leq n$ be the integer such that an expected occurrence of $v_{3m+3,j}$ in $x$ contains the position $0$. Suppose $h$ is the smallest positive integer with $\sigma^h(x)\in A_m$. Then there is an expected occurrence of $v_{3m+3,j'}$ in $x$ for some $1\leq j'\leq n$ such that the second expected occurrence of $v_{3m,1}$ in this occurrence of $v_{3m+3, j'}$ starts exactly at $h-\sum_{0\le i\le m-1}k_{3i}$. By the minimality of $h$, we get that the expected occurrence of $v_{3m+3,j}$ and this expected occurrence of $v_{3m+3,j'}$ are only separated by some $1^s$. Conversely, the expected occurrence of some $v_{3m+3,j'}$ immediately to the right of the expected occurrence of $v_{3m+3,j}$ determines the smallest $h$ such that $\sigma^h(x)\in A_m$. Therefore, for $1\le j,j'\le n$ and $0\le s\le\ell$, if we let $B_{j,s,j'}$ be the set of all $x\in X_V$ such that there is an expected occurrence of $v_{3m,1}$ in $x$ starting at the position $-\sum_{0\le i\le m-1}k_{3i}$, which is the second expected occurrence of $v_{3m,1}$ in an expected occurrence of $v_{3m+3,j}$ in $x$, and this expected occurrence of $v_{3m+3,j}$ is followed by $1^s$ and an expected occurrence of $v_{3m+3,j'}$ in $x$, we know that $\{B_{j,s,j'}:1\le j,j'\le n;0\le s\le\ell\}$ is a clopen partition of $A_m$ and this partition refines $\mathcal{P}\upharpoonright A_m$. In summary, we obtain a new Kakutani--Rohlin partition $\mathcal{P}'$ whose base is still $A_m$, and if $B_{j,s,j'}\ne \varnothing$, then $B_{j,s,j'}\in\mathcal{P}'$. $\mathcal{P}'$ has at most $n^2\ell$ towers.

Finally note that the diameter of $A_m$ is at most $2^{-|v_{3m-3,1}|}$ since for any $x\in A_m$, $x\rest [-|v_{3m-3,1}|, |v_{3m-3,1}|]=x_0\rest[-|v_{3m-3,1}|, |v_{3m-3,1}|]$ and is thus completely fixed. Similarly, every clopen set in $\mathcal{P}'$ has a diameter at most $2^{-|v_{3m-3,1}|}$. By Theorem~\ref{thm:3.3} (2), $X_V$ has topological rank at most $n^2\ell$.
\end{proof}

\begin{theorem}\label{thm:6.7} Let $X$ be an infinite minimal subshift of finite symbolic rank. Then $X$ has finite topological rank.
\end{theorem}

\begin{proof} By Corollary~\ref{cor:4.3} $X=X_V$ where $V$ has a proper rank-$n$ construction for some $n\geq 1$. By Proposition~\ref{prop:6.3} there is a word $W$ which has a good rank-$2n$ construction such that $X_V$ is a factor of $X_W$ and $X_W$ is minimal. By Proposition~\ref{prop:6.5}, $X_W$ has finite topological rank. Thus $X_V$ has finite topological rank by the main theorem (Theorem~1.1) of \cite{GH*}.
\end{proof}

In \cite{GH*} the authors showed that if a minimal Cantor system $(Y,S)$ is a factor of a minimal Cantor system $(X, T)$ of finite topological rank, then $\toprank(Y,S)\leq 3\toprank(X,T)$. In \cite{Es} Corollary 4.8 this is improved to $\toprank(Y,S)\leq \toprank(X,T)$. Combining these with our results, we can state the following quantitative result.

\begin{corollary}\label{cor:6.8} Let $X_V$ be an infinite minimal subshift of finite symbolic rank. Then $$\toprank(X_V,\sigma)\leq 4(M+1)(\symbrank(X_V))^2 $$ 
where $M$ is a bound for the spacer parameter of any proper construction of $V$.
\end{corollary}

\subsection{From finite topological rank to finite symbolic rank\label{subsec:6.2}} 

It was proved in \cite{DM} that every minimal Cantor system of finite topological rank is either an odometer or a subshift on a finite alphabet. We will show that in case it is a subshift it is conjugate to a subshift of finite symbolic rank.

We use the following notation from \cite{GH*} (with a slight modification) in this subsection. Let $B=(V, E, \preceq)$ be an ordered Bratteli diagram. For each $i\geq 1$, let $V_i^*$ denote the set of all words on the alphabet $V_i$, and define a map $\eta_{i+1}: V_{i+1}\to V_i^*$ as follows. For $v\in V_{i+1}$, enumerate all edges $e\in E_{i+1}$ with $\mathsf{r}(e)=v$ in the $\preceq$-increasing order as $e_1,\dots, e_k$, and define
$$ \eta_{i+1}(v)=\mathsf{s}(e_1)\cdots \mathsf{s}(e_k). $$
We also define $\eta_1: V_1\to E_1^*$, where $E_1^*$ is the set of all words on the alphabet $E_1$. For $v\in V_1$, enumerate all $e\in E_1$ with $\mathsf{r}(e)=v$ in the $\preceq$-increasing order as $e_1\dots e_k$, and define
$$ \eta_1(v)=e_1\cdots e_k. $$

\begin{theorem} \label{thm:6.9} Every minimal Cantor system $(X,T)$ of finite topological rank is an odometer or is conjugate to a minimal subshift $X_V$ of finite symbolic rank. Moreover, if $(X,T)$ is not an odometer, then $\symbrank(X_V)\leq \toprank(X,T)$.
\end{theorem}

\begin{proof}
We just need to show that for every simple ordered Bratteli diagram $B=(W,E,\preceq)$ where $|W_i|\le n$ for all $i\geq 1$, if the Bratteli-Vershik system $(X_B,\lambda_B)$ generated by $B$ is not an odometer, then it is conjugate to $X_V$ for a word $V$ which has a rank-$n$ construction. By telescoping if necessary, we assume without loss of generality that the following properties hold for $B$:
\begin{enumerate}
\item[(1)] for each $i\geq 0$, $w\in W_i$ and $w'\in W_{i+1}$, there is an edge $e\in E_{i+1}$ with
$\mathsf{s}(e)=w$ and $\mathsf{r}(e)=w'$;
\item[(2)] for each $i\geq 1$, $|W_i|\ge 2$;
\item[(3)] for each $i\geq 1$, there are vertices $w^i_{\min}$ and $w^i_{\max}$ in $W_i$ such that for every $w\in W_{i+1}$, $\eta_{i+1}(w)$ starts with $w^i_{min}$ and ends with $w^i_{max}$;
\item[(4)] for each $w\in W_1$, $|\eta_1(w)|\gg n$;
\item[(5)] for any $x,y\in X_B$, if $x\ne y$, then there exists $w\in W_1$ such that $\Ret_{A_w}(x)\ne \Ret_{A_w}(y)$, where $A_w$ denotes the union of the Kakutani--Rohlin tower determined by $w$. 
\end{enumerate}

For (3), we consider the unique $\xmin$ and $\xmax$. Fix an $i\geq 1$. Let $w^i_{\min}\in W_i$ be the vertex in $W_i$ which $\xmin$ passes through and $w^i_{\max}\in W_i$ be the vertex in $W_i$ which $\xmax$ passes through. Then by the uniqueness of $\xmin$, there is an $i_0>i$ such that for all $i'\geq i_0$ and $w\in W_{i'}$ the minimal path between $v_0$ and $w$ passes through $w^i_{\min}$. Similarly, there is an $i_1$ such that for all $i'\geq i_1$ and $w\in W_{i'}$, the maximal path between $v_0$ and $w$ passes through $w^i_{\max}$. Now we get (3) by telescoping. 

For (5) we use the main theorem of \cite{DM}, which guarantees that $(X,T)$ is a subshift on a finite alphabet. In particular there is a finite partition $\mathcal{P}$ of $X$ into clopen sets such that the smallest Boolean algebra containing elements of $\mathcal{P}$ and closed under $T$ and $T^{-1}$ contains all clopen subsets of $X$. Now we also have that $(X, T)$ is conjugate to $(X_B,\lambda_B)$. Thus there is also a finite partition $\mathcal{Q}$ of $X_B$ into clopen sets such that the smallest Boolean algebra containing elements of $\mathcal{Q}$ and closed under $\lambda_B$ and $\lambda_B^{-1}$ contains all clopen subsets of $X_B$. Hence there is $i\geq 1$ such that every element of $\mathcal{Q}$ is the union of the basic open sets given by the paths from $v_0$ to some elements of $W_i$. Let $F$ be the set of all paths from $v_0$ to an element of $W_i$. For each $p\in F$ let $N_p$ denote the basic open set of $X_B$ given by $p$. Then for all $x, y\in X_B$ with $x\neq y$, there is $p\in F$ such that $\Ret_{N_p}(x)\neq \Ret_{N_p}(y)$. Now for any $w\in W_i$, let $A_w$ denote the union of the Kakutani--Rohlin tower determined by $w$, then $A_w$ is the clopen set given by all paths from $v_0$ to $w$. We claim that for all $x, y\in X_B$ with $x\neq y$, there is $w\in W_i$ such that $\Ret_{A_w}(x)\neq \Ret_{A_w}(y)$. For this, note that for any $w\in W_i$, if we enumerate all paths from $v_0$ to $w$ in the $\preceq'$-increasing order as $p_1,\dots, p_k$, then for any $x\in X_B$, $1\leq j\leq k$ and $m\in\mathbb{Z}$, $m\in \Ret_{N_{p_j}}(x)$ iff $m-ak-j+1,\dots, m\in \Ret_{A_w}(x)$ and $m-ak-j\notin \Ret_{A_w}(x)$ for a natural number $a$. Thus if $x\neq y\in X_B$ and $p$ is a path from $v_0$ to $w$ such that $\Ret_{N_p}(x)\neq\Ret_{N_p}(y)$, then $\Ret_{A_w}(x)\neq\Ret_{A_w}(y)$. Now (5) follows by telescoping.
 
For each $i\geq 1$, enumerate the elements of $W_i$ as $w_{i,1},w_{i,2},\cdots,w_{i,n_i}$, where $2\le n_i\le n$, so that $w_{i,1}=w^i_{\min}$. Define
$$ v_{1,j}=0(01)^j0^{|\eta_1(w_{1,j})|-2n-4j-2}(10)^{j+n}0 $$
for $1\le j\le n_1$. For $i\ge 2$, assume $v_{i-1,j}$ have been defined for all $1\le j\le n_{i-1}$. Then we define 
$$ v_{i,j}=v_{i-1,j_1}v_{i-1,j_2}\cdots v_{i-1,j_k}$$ 
if 
$$\eta_i(w_{i,j})=w_{i-1,j_1}w_{i-1,j_2}\cdots w_{i-1,j_k}.$$ 
It is clear that this defines a rank-$n$ construction. Let $V=\lim_i v_{i,1}$.

We note that $V$ has a proper rank-$m$ construction for some $m\leq n$. In fact, let $m\geq 2$ be the smallest such that $n_i=m$ for infinitely many $i\geq  2$. Let $\{i_k\}_{k\geq 0}$ enumerate this infinite set of indices. Then by telescoping with respect to $\{i_k\}_{k\geq 0}$, we obtain a proper rank-$m$ construction for $V$. We note in addition that the rank-$m$ generating sequence associated to this construction is a subsequence of the rank-$n$ generating sequence $\{v_{i,j}\}_{i\geq 0, 1\leq j\leq n_i}$, and therefore has bounded spacer parameter. Thus by Corollary~\ref{cor:4.3} $X_V$ is minimal.

We claim that for any $1\le j\le n_1$, $v_{1,j}$ is not of the form 
$$ \alpha 1^{s_1}v_{1,j_1}1^{s_2}\cdots v_{1,j_{k-1}}1^{s_k}\beta $$
where $k> 0$, $\alpha$ is a nonempty suffix of some $v_{1,j_k}$, and $\beta$ is a nonempty prefix of some $v_{1,j_{k+1}}$. This follows easily from the observation that $v_{1,j}$ has a prefix  of the form $00(10)^j00$ and a suffix of the form $00(01)^{j+n}00$, and for $1\leq j'\leq n_1$ where $j'\neq j$, $v_{1,j'}$ does not contain either of these words as a subword. 

Let $Y=(2^{W_1})^{\mathbb{Z}}$ and view it as a shift over the alphabet $2^{W_1}$. Define $\theta: X_B\to Y$ by 
$$ \theta(x)(k)(w)=\Ret_{A_w}(x)(k). $$
Then $\theta$ is clearly continuous. It is easy to check that $\theta\circ \lambda_B=\sigma\circ \theta$. By (5), $\theta$ is injective. Thus $\theta$ is a conjugacy map between $(X_B, \lambda_B)$ and $\theta(X_B)$, which is a subshift of $Y$.

Finally, we verify that $X_V$ is conjugate to $\theta(X_B)$. For this we define $\varphi: X_V\to (2^{W_1})^{\mathbb{Z}}$ by letting $\varphi(z)(k)(w_{1,j})=1$ iff there is $k'\leq k$ with $k'+|v_{1,j}|-1\geq k$ such that $v_{1,j}$ occurs in $z$ starting at position $k'$. $\varphi$ is well defined because of the above claim. It is clear that $\varphi$ is continuous and injective, and $\varphi\circ \sigma=\sigma\circ \varphi$. Thus $\varphi$ is a conjugacy map between $X_V$ and $\varphi(X_V)$. To complete our proof, it suffices to show $\theta(X_B)=\varphi(X_V)$. Consider a $y\in X_V$ such that $y\rest [0,\infty)=V$. Then by our definitions of $\theta$ and $\varphi$, and particularly because $|v_{i,j}|=|\eta(w_{1,j})|$ for all $1\leq j\leq n_1$, we have $\theta(\xmin)\rest{[0,\infty)}=\varphi(y)\rest [0,\infty)$. By the shift-invariance and the compactness of $\theta(X_B)$ and $\varphi(X_V)$, we get $\theta(X_B)\cap\varphi(X_V)\ne\varnothing$. By the minimality of $\theta(X_V)$ and $\varphi(X_B)$, we have $\theta(X_B)=\varphi(X_V)$ as required.
\end{proof}

The consideration of the shift $Y$ in the above proof is motivated by the work in \cite{Ka} (the construction before Theorem 3.4 in \cite{Ka}).

\subsection{Some examples}

In this subsection we give some examples to demonstrate that the results in the preceding subsections are optimal. We first show that a non-minimal rank-$1$ subshift need not have finite topological rank.

\begin{proposition}\label{prop:6.10} There exists a rank-$1$ word $V$ such that $X_V$ is not minimal and $X_V$ is not of finite topological rank.
\end{proposition}

\begin{proof} For any $n\geq 0$, let $r_n\geq 2n+5$ and $s_{n,1}, \dots, s_{n, r_n-1}$ be nonnegative integers satisfying the following:
\begin{enumerate}
\item[(i)] $s_{n,1}=3n+1$, and $s_{n, r_n-1}=3n+2$;
\item[(ii)] for all $1<i<r_n-1$, $s_{n,i}=3m$ for some $0\leq m\leq n+1$;
\item[(iii)] for any $1\le m\le n+1$, there exist $1<i<r_n-1$ such that $s_{n,i}=s_{n,i+1}=3m$.
\end{enumerate}
Then as usual, define $v_0=0$ and $v_{n+1}=v_n1^{s_{n,1}}v_n1^{s_{n,2}}\cdots 1^{s_{n,r_n-1}}v_n$ inductively, and let $V=\lim_n v_{n,1}$. 


We note that for any $n\ge1$, $0\leq m\leq n+1$ and $u$ a nonempty prefix of $v_n$, $01^{3m}u$ is not a suffix of $v_n$. 


Toward a contradiction, assume $X_V$ has topological rank $K\geq 1$. Fix a positive integer $N\geq 1$. Then by Theorem~\ref{thm:3.3} there is a Kakutani--Rohlin partition  $\mathcal{P}$ of $X_V$ with the following properties:
\begin{enumerate}
\item[(a)] $\mathcal{P}$ has $K$ many towers, with bases $B_1,\dots, B_K$; 
\item[(b)] $1^{\mathbb{Z}}\in B(\mathcal{P})=\bigcup_{1\leq k\leq K} B_k$ and $\diam(B(\mathcal{P}))<2^{-N-2}$;
\item[(c)] $\diam(A)<2^{-N-2}$ for all $A\in\mathcal{P}$.
\end{enumerate}
Since every $A\in\mathcal{P}$ is clopen, there exists $N'>N+4$ such that for every $A\in\mathcal{P}$, there exists $U_A\subseteq \{0,1\}^{2N'+1}$ with 
$$A=\{x\in X_V\,:\, x\rest [-N', N']\in U_A\}.$$

Let $n\gg N'+3N$. 

Fix any $1\leq m\leq n+1$. Let $x\in X_V$ be such that $v_n1^{3m}v_n1^{3m}v_n$ occurs in $x$ at position $-|v_n|$, and each of the three demonstrated occurrences of $v_n$ is expected. Then from the definition of $N'$, and because $|v_n|\geq 2^n\gg n\gg N'$, we have that $\sigma^{|v_n|+3m}(x)$ and $x$ belong to the same set in the partition $\mathcal{P}$. Thus there is $0< j\leq |v_n|+3m$ such that $\sigma^j(x)\in B(\mathcal{P})$. Let $t=3m-j$. Then $-|v_n|\leq t< 3m$, $y\triangleq \sigma^{3m-t}(x)\in B_{k}$ for some $1\leq k\leq K$, $v_n1^{3m}v_n1^{3m}v_n$ occurs in $y$ at position $t-|v_n|-3m$, and every $z\in X_V$ with an occurrence of $v_n1^{3m}v_n1^{3m}v_n$ at position $t-|v_n|-3m$ is in $B_{k}$. Let $t_m$ be the least such $t$ and $k_m$ be the corresponding $k$. We note the following two properties of the element $y$. First, there is an occurrence of $v_n1^{3m}v_n$ in $y$ at position $t_m$. Second, because of the minimality of $t_m$, we have that for any $0\leq j\leq t_m+|v_n|$, $\sigma^j(y)\in \sigma^j(B_{k_m})$ and $\sigma^j(B_{k_m})\cap B(\mathcal{P})=\varnothing$ when $j\ne 0$, and so $\sigma^j(B_{k_m})$ is an element of $\mathcal{P}$ (it is one of the sets in the $k_m$-th tower of $\mathcal{P}$).

We claim that for any $1\leq m_1, m_2\leq \lfloor N/3\rfloor$ with $m_1\neq m_2$, we must have $k_{m_1}\neq k_{m_2}$. Toward a contradiction, assume $k\triangleq k_{m_1}=k_{m_2}$. Without loss of generality assume $t_{m_1}\leq t_{m_2}$. 

Consider first the case $t_{m_1}<t_{m_2}$. Then we have a subclaim that $t_{m_2}\geq t_{m_1}+3m_1$. To see this, let $y_1\in B_k$ be an element with an occurrence of $v_n1^{3m_1}v_n$ at position $t_{m_1}$ as above, and similarly $y_2\in B_k$ be an element with an occurrence of $v_n1^{3m_2}v_n$ at $t_2$. Since $\diam(B_k)<2^{-N-2}$, $y_1\rest [-N,N]=y_2\rest[-N,N]$. Also, since for all $0\leq j\leq t_{m_2}+|v_n|$, $\sigma^j(B_k)$ is an element of $\mathcal{P}$, which has diameter $<2^{-N-2}$, we have that $y_1\rest[-N+j, j+N]=y_2\rest[-N+j, j+N]$ for all $0\leq j\leq t_{m_2}+|v_n|$. In particular $y_2(t_{m_2}+|v_n|-1)=0=y_1(t_{m_2}+|v_n|-1)$. Since $t_{m_2}+|v_n|-1\geq t_{m_1}+|v_n|$ and $y_1$ has an occurrence of $1^{3m_1}$ at $t_{m_1}+|v_n|$, we must have $t_{m_2}+|v_n|-1\geq t_{m_1}+|v_n|+3m_1-1$, or $t_{m_2}\geq t_{m_1}+3m_1$ as in the subclaim. Note that our argument above gives that 
$$y_1\rest [t_{m_1}+|v_n|-1, t_{m_2}+|v_n|-1]=y_2\rest [t_{m_1}+|v_n|-1, t_{m_2}+|v_n|-1].$$
Since $t_{m_2}\geq t_{m_1}+3m_1$, the left-hand side is a word of the form $01^{3m_1}u$ where $u$ is a nonempty prefix of $v_n$. But the right-hand side is a suffix of $v_n$. This contradicts our construction of $v_n$.

Thus $t_{m_1}=t_{m_2}$. Denote $t\triangleq t_{m_1}=t_{m_2}$. Without loss of generality assume $m_1<m_2$. By the above argument we again have $y_1\rest [-N+t+|v_n|, t+|v_n|+N]=y_2\rest [-N+t+|v_n|, t+|v_n|+N]$. Since $3m_1<3m_2\leq N$, we have in particular $y_1\rest[t+|v_n|, t+|v_n|+3m_2-1]=y_2\rest[t+|v_n|, t+|v_n|+3m_2-1]$. But the left-hand side is of the form $1^{3m_1}u$ where $u$ is a nonempty prefix of $v_n$, while the right-hand side is $1^{3m_2}$, a contradiction.

This finishes our proof of the claim that whenever $1\leq m_1\neq m_2\leq \lfloor N/3\rfloor$, we have $k_{m_1}\ne k_{m_2}$. It follows from the claim that $K\geq \lfloor N/3\rfloor$. This contradicts the arbitrariness of $N$.
\end{proof}

The next examples show that the topological rank is not bounded by a function of the symbolic rank alone, thus the extra parameter as in Corollary~\ref{cor:6.8} is necessary. The proposition is also a consequence of \cite[Corollary 4.9]{AD}.

\begin{proposition}\label{prop:6.11} For any $N>1$, there is a minimal rank-$1$ subshift whose topological rank is at least $N$.
\end{proposition}

\begin{proof} Fix $p\geq 2N$ and $q\gg N$. Define $v_0=0$ and 
$$ v_{n+1}=(v_n1)^{q}v_n1^{a_{n,1}}v_n1^{a_{n,2}}\cdots 1^{a_{n,p}}v_n(1^{N+2}v_n)^{q}, $$ 
where $a_{n,1}, \dots, a_{n,p}$ are nonnegative integers satisfying the following:
\begin{enumerate}
\item[(i)] for any $1\leq i\leq p$, $2\leq a_{n,i}\leq N+1$;
\item[(ii)] for any $2\leq m\leq N+1$, there is $1\leq i\leq p$ such that $a_{n,i}=a_{n,i+1}=m$.
\end{enumerate}
Let $V=\lim_{n} v_n$. 

By an easy induction we have that for all $n\geq 1$ and $1\leq m\leq N+1$, if $u$ is a nonempty prefix of $v_n$, then $01^{m}u$ is not a suffix of $v_n$.

Consider a Kakutani--Rohlin partition $\mathcal{P}$ of $X_V$ such that 
\begin{enumerate}
\item[(a)] $\mathcal{P}$ has $K$ many towers, with bases $B_1,\cdots,B_K$;
\item[(b)] $\diam(B(\mathcal{P}))<2^{-N-4}$;
\item[(c)] $\diam(A)<2^{-N-4}$ for all $A\in\mathcal{P}$.
\end{enumerate}
Since every $A\in\mathcal{P}$ is clopen, there exists $N'>N+6$ such that for every $A\in\mathcal{P}$, there exists $U_A\subset \{0,1\}^{2N'+1}$ with
$$ A=\{x\in X_V\,:\, x\rest[-N',N']\in U_A\}. $$

Let $n\gg N'+3N$. Similar to the proof of Proposition~\ref{prop:6.10}, we can define, for each $2\leq m\leq N+1$, numbers $t_m$ where $-|v_n|\leq t_m\leq m$, $k_m$ where $1\leq k_m\leq K$, and an element $y\in B_{k_m}$ such that $v_n1^mv_n$ occurs in $y$ at position $t_m$ and for all $0\leq 1\leq t_m+|v_n|$, $\sigma^q(B_{k_m})$ is an element of $\mathcal{P}$.

As in the proof of Proposition~\ref{prop:6.10}, we have that for all $2\leq m_1,m_2\leq N+1$, if $m_1\neq m_2$, then $k_{m_1}\neq k_{m_2}$. This implies that $K\geq N$.
\end{proof}


\subsection{From finite alphabet rank to finite symbolic rank\label{subsec:6.4}}
In this subsection we explore some connections between subshifts of finite symbolic rank and $\mathcal{S}$-adic subshifts of finite alphabet rank considered by various authors, e.g. \cite{BSTY} and \cite{DDMP}.

We first recall the basic definition of $\mathcal{S}$-adic subshifts and related notions following \cite{DDMP}. For a finite alphabet $A$, let $A^*$ be the set of all finite words on $A$. If $A, B$ are finite alphabets, a {\em morphism} $\tau: A^*\to B^*$ is a map satisfying that $\tau(\varnothing)=\varnothing$ and for all $u, v\in A^*$, $\tau(uv)=\tau(u)\tau(v)$. A {\em directive sequence} is a sequence of morphisms $\boldsymbol{\tau}=(\tau_n\,:\, A_{n+1}^*\to A_n^*)_{n\geq 0}$. For $0\leq n<N$, denote by $\tau_{[n,N)}$ the morphism $\tau_n\circ\tau_{n+1}\circ\cdots\circ \tau_{N-1}$. For any $n\geq 0$, define
$$ L^{(n)}(\boldsymbol{\tau})=\{w\in A_n^*\,:\, \mbox{$w$ occurs in $\tau_{[n,N)}(a)$ for some $a\in A_N$ and $N>n$}\} $$
and
$$ X^{(n)}_{\boldsymbol{\tau}}=\{x\in A_n^\mathbb{Z}\,:\, \mbox{every finite subword of $x$ is a subword of some $w\in L^{(n)}(\boldsymbol{\tau})$}\}. $$
$X^{(n)}_{\boldsymbol{\tau}}$ is a subshift on the alphabet $A_n$, and we denote the shift map by $\sigma$. Now let $X_{\boldsymbol{\tau}}=X^{(0)}_{\bf\tau}$. Then $(X_{\boldsymbol{\tau}},\sigma)$ is the {\em $\mathcal{S}$-adic subshift} generated by the directive sequence $\boldsymbol{\tau}$. The {\em alphabet rank} of $\boldsymbol{\tau}$ is defined as
$$ \mbox{AR}(\boldsymbol{\tau})=\liminf_{n\to\infty}|A_n| $$
and the {\em alphabet rank} of a subshift $(X,\sigma)$ as
$$ \mbox{AR}(X)=\inf\{\mbox{AR}(\boldsymbol{\tau})\,:\, X_{\boldsymbol{\tau}}=X\}. $$
As a convention, $\inf\varnothing=+\infty$.

There is a similar notion of {\em telescoping} for directive sequence $\boldsymbol{\tau}$ which does not change the $\mathcal{S}$-adic subshift generated by $\boldsymbol{\tau}$.

An $\mathcal{S}$-adic subshift $X_{\boldsymbol{\tau}}$ is {\em primitive} if for any $n\geq 0$ there exists $N>n$ such that $\tau_{[n,N)}(a)$ contains all letters in $A_n$ for all $a\in A_N$. 

If $\tau:A^*\to B^*$ is a morphism, $x\in B^\mathbb{Z}$, and $Y\subseteq A^\mathbb{Z}$ is a subshift, then a {\em $\tau$-representation} of $x$ in $Y$ is a pair $(k,y)\in \mathbb{Z}\times Y$ such that $x=\sigma^k(\tau(y))$. Moreover, $(k,y)$ is a {\em centered} $\tau$-representation if $0\leq k<|\tau(y(0))|$ in addition. Now $\tau$ is {\em recognizable in $Y$} if each $x\in B^\mathbb{Z}$ has at most one centered $\tau$-representation in $Y$, and a directive sequence $\boldsymbol{\tau}=(\tau_n:A_{n+1}^*\to A_n^*)_{n\geq 0}$ is {\em recognizable} if for each $n\geq 0$, $\tau_n$ is recognizable in $X_{\boldsymbol{\tau}}^{(n+1)}$. An $\mathcal{S}$-adic subshift $X_{\boldsymbol{\tau}}$ is {\em recognizable} if $\boldsymbol{\tau}$ is recognizable.

\begin{theorem}\label{thm:6last} Let $X_{\boldsymbol{\tau}}$ be a primitive, recognizable $\mathcal{S}$-adic subshift of finite alphabet rank $K$. Then $(X_{\boldsymbol{\tau}},\sigma)$ is conjugate to a subshift of finite symbolic rank $\leq K$. Moreover, there exists a proper rank-$K$ construction for a uniquely readable rank-$K$ generating sequence $\{v_{i,j}\}_{i\geq 0, 1\leq j\leq K}$ such that $(X_{\boldsymbol{\tau}},\sigma)$ is conjugate to $(X_V,\sigma)$, where $V=\lim_{i}v_{i,1}$.
\end{theorem}

\begin{proof} This is similar to the proof of Theorem~\ref{thm:6.9}. By telescoping if necessary, we assume without loss of generality that the following properties hold for $\boldsymbol{\tau}$:
\begin{enumerate}
\item[(1)] for each $i\geq 0$, $a\in A_i$ and $b\in A_{i+1}$, $\tau_i(b)$ contains the letter $a$;
\item[(2)] for each $i\geq 1$, $|A_i|=K$;
\item[(3)] for each $a\in A_1$, $|\tau_0(a)|\gg K$.
\end{enumerate}
Since each $A_i$ is finite, a finite splitting argument similar to the proof of Proposition~\ref{prop:4.1} shows that we can enumerate each $A_i$ as $a_{i,1},\dots, a_{i,n_i}$ such that $n_i=K$ for all $i>1$ and for each $i\geq 0$, $\tau_i(a_{i+1,1})$ starts with $a_{i,1}$. Now, as in the proof of Theorem~\ref{thm:6.9}, define
$$ v_{1,j}=0(01)^j0^{|\tau_0(a_{1,j})|-2n-4j-2}(10)^{j+n}0 $$
for $1\leq j\leq K$. For $i\geq 1$ and $1\leq j\leq K$, if
$$ \tau_{i}(a_{i+1,j})=a_{i,j_1}a_{i,j_2}\cdots a_{i,j_k}, $$
then let
$$ v_{i+1,j}=v_{i,j_1}v_{i,j_2}\cdots v_{i,j_k}. $$
This gives a proper rank-$K$ construction for $V=\lim_iv_{i,1}$.

Clearly the recognizability of $\boldsymbol{\tau}$, together with our definition of $v_{1,j}$, imply the unique readability of $\{v_{i,j}\}_{i\geq 0, 1\leq j\leq K}$. Now $(X_{\boldsymbol{\tau}},\sigma)$ and $(X_V,\sigma)$ are conjugate by the substitution $\tau_0(a_{1,j})\mapsto v_{1,j}$.
\end{proof}

With Theorem~\ref{thm:6last}, our Theorem~\ref{thm:6.9} becomes a consequence of the main theorem of \cite{DDMP} which states that every minimal Cantor system of finite topologival rank  is either an odometer or conjugate to a primitive, recognizable $\mathcal{S}$-adic subshift of finite alphabet rank.

\section{Density and genericity of subshifts of finite symbolic rank\label{sec:new6}}

 It is known that the set of rank-$1$ measure-preserving transformations is a dense $G_\delta$ subset of the Polish space of all measure-preserving transformations (see \cite{Fe}). Here in the topological setting, we show that the situation is different. In fact we consider various different spaces of Cantor systems and subshifts and show that the class of all rank-$1$ subshifts is dense in all but one  of them but generic in none. On the other hand, we note that subshifts of symbolic rank $2$ are generic in the spaces for all transitive and totally transitive subshifts.

We start with the coding space for all minimal Cantor systems.

\begin{proposition}\label{prop:5.4} The set of all minimal Cantor systems conjugate to a rank-$1$ subshift is dense but not generic in the space of all minimal Cantor systems.
\end{proposition}

\begin{proof} By Proposition~\ref{prop:odometerdense} and Lemma~\ref{lem:2.4}, the set of all subshifts is meager, and not generic, in the space of all minimal Cantor systems. For the density, in view of Proposition~\ref{prop:odometerdense}, it suffices to show that infinite minimal rank-$1$ subshifts can approximate any infinite odometer. To be precise, we need to show that for all $k\geq 2$ there is an infinite rank-$1$ subshift $X_V$ and a clopen subset $A$ of $X_V$ such that $\sigma^k(A)=A$ and $\{A, \sigma(A), \dots, \sigma^{k-1}(A)\}$ form a partition of $X_V$.

Fix $k\geq 2$. We define the following {\em Chac\'{o}n-like} rank-$1$ generating sequence:
$$ \begin{array}{rcl} v_0&=& 0 \\
v_1&=&0^{2k}1^k0^k \\
v_{n+1}&=& v_nv_n1^kv_n \mbox{ for $n\geq 1$.}
\end{array}
$$
Let $V=\lim_n v_n$. Then $X_V$ is infinite. Let $A$ be the set of all $x\in X_V$ such that
$$ x\rest[0,3k-1]\in\{0^{3k}, 0^{2k}1^k, 0^k1^k0^k, 1^k0^k1^k, 1^k0^{2k}\}. $$
Then $A$ is a clopen subset of $X_V$ with the required property.
\end{proof}

To investigate the density and the genericity of subshifts of finite symbolic rank, we consider some spaces of subshifts as defined in \cite{PS}. First, let $\mathcal{S}_2$ be the space of all $\sigma$-invariant closed subsets of $2^\mathbb{Z}$. $\mathcal{S}_2$ is a $G_\delta$ subspace of $K(2^\mathbb{Z})$, and hence is a Polish space. The Hausdorff metric on $\mathcal{S}_2$ is equivalent to the following metric which is easier to work with in our setting. For $X\in \mathcal{S}_2$ and integer $n\geq 0$, let $L_n(X)$ be the set of all finite words of length $n$ which occurs in some element of $X$. Let $L(X)=\bigcup_n L_n(X)$. For $X, Y\in\mathcal{S}_2$, let
$$ d_L(X,Y)=2^{-\inf\{n\,:\, L_n(X)\neq L_n(Y)\}}. $$

However, $\mathcal{S}_2$ is not a perfect space; in particular, the finite subshifts are isolated points in this space. Thus, following \cite{PS} we consider the following perfect subspace, which in particular includes all infinite rank-$1$ subshifts.  Let $\mathcal{S}'_2$ be the subspace of all elements of $\mathcal{S}_2$ which are not isolated (the notation is inspired by the Cantor--Bendixson derivative; see \cite{Ke}). $\mathcal{S}'_2$ is a perfect subspace of $\mathcal{S}_2$, and hence a Polish space. 

Recall that a Cantor system $(X, T)$ is {\em (point) transitive} if there exists $x\in X$ such that the orbit of $x$ is dense in $X$; it is {\em totally (point) transitive} if for all integer $n\geq 1$ there exists $x\in X$ such that $\{T^{nk}x\,:\, k\in\mathbb{Z}\}$ is dense in $X$.

Let $\mathcal{T}'_2$ be the subspace of all transitive subshifts in $\mathcal{S}'_2$. Let $\overline{\mathcal{T}'_2}$ be the closure of $\mathcal{T}'_2$ in $\mathcal{S}'_2$. Moreover, let $\mathcal{TT}'_2$ be the subspace of all totally transitive subshifts in $\mathcal{S}'_2$. Let $\overline{\mathcal{TT}'_2}$ be the closure of $\mathcal{TT}'_2$ in $\mathcal{S}'_2$. Then $\overline{\mathcal{TT}'_2}\subseteq \overline{\mathcal{T}'_2}$ are both closed subspaces of $\mathcal{S}'_2$, hence are Polish spaces, and the metric $d_L$ remains a compatible metric on these subspaces.

The following theorem shows that minimal rank-$1$ subshifts can approximate infinite minimal subshifts of topological rank $2$ in the sense of $d_L$.

\begin{theorem}\label{thm:new7.2}  Let $n\geq 1$ and let $(X, \sigma)$ be an infinite minimal subshift of topological rank $2$. Then there exists an infinite minimal subshift $(Y,\sigma)$ such that $L_n(X)=L_n(Y)$ and $(Y,\sigma)$ is conjugate to $(X_V,\sigma)$ for some infinite rank-1 word $V$. Moreover, if $(X,\sigma)$ is totally transitive, then we can find $(Y,\sigma)$ which is also totally transitive.
\end{theorem}

\begin{proof} By the main theorem of \cite{DDMP}, $(X,\sigma)$ is conjugate to a primitive, recognizable $\mathcal{S}$-adic subshift of alphabet rank 2. By the proof of Theorem~\ref{thm:6last}, there exists a proper rank-$2$ construction for an infinite word $W$ with the following properties:
\begin{enumerate}
\item[(1)] the associated rank-$2$ generating sequence $\{w_{i,j}\}_{i\geq 0,1\le j\le2}$ has unique readability;
\item[(2)] for all $i\geq 1$ and $j=1,2$, the spacer parameter in the building of $w_{i+1,j}$ from $\{w_{i,1},w_{i,2}\}$ is bounded by 0;
\item[(3)] $(X_W, \sigma)$ is conjugate to $(X,\sigma)$.
\end{enumerate}
Let $f$ be a conjugacy map from $(X_W,\sigma)$ to $(X,\sigma)$. 

Fix $n_1\ge1$ such that for any $x,y\in X_W$ and $k\in\mathbb{Z}$, whenever $x\rest{[k-n_1,k+n_1]}=y\rest{[k-n_1,k+n_1]}$, we have $f(x)(k)=f(y)(k)$. For any $v\in L(X_W)$, if $|v|>2n_1$ and for some $x\in X_W$ and $k\in\mathbb{Z}$, we have $x\rest{[k,k+|v|-1]}=v$, then define $\Phi(v)=f(x)\rest {[k+n_1,k+|v|-n_1-1]}$. Clearly $\Phi(v)$ is well defined and does not depend on the choice of $x$.  

For any finite or infinite word $u$ and $m\leq |u|$, let $L_m(u)$ denote the set of all subwords of $u$ of length $n$. 

Let $i_0\ge1$ be sufficiently large such that for $j=1,2$, $|w_{i_0,j}|>2n+4n_1$ and $L_{n+2n_1}(W)=L_{n+2n_1}(w_{i_0,j})$. Since $X$ is infinite, $W$ is aperiodic, and it follows that there is $j_0\in \{1,2\}$ such that for any $j=1,2$, both $w_{i_0,j_0}w_{i_0,j}\in L(W)$ and $w_{i_0,j}w_{i_0,j_0}\in L(W)$. For the same reason, there exists $i_1>i_0$ such that $\Phi({w_{i_1,1}})$ does not have a period $t$ for any $t\le |w_{i_0,1}|+|w_{i_0,2}|$, i.e., there are $0\leq a< a+kt<|\Phi(w_{i_1,1})|$ such that $\Phi(w_{i_1,1})(a)\neq \Phi(w_{i_1,1})(a+kt)$.

Define a rank-1 generating sequence by letting 
$$v_1=w_{i_0,j_0}w_{i_1,1}w_{i_0,j_0}$$ 
and for any $i\ge1$, 
$$v_{i+1}=v_iv_i1^{|w_{i_0,j_0}|}v_i. $$ 
As usual, let $V=\lim_i v_i$. Then $V$ is a minimal aperiodic infinite rank-1 word.

Define a map $g$ from $X_V$ to $2^\mathbb{Z}$ as follows. For $x\in X_V$, if $k$ is a part of an expected occurrence of $v_1$ in $x$, then set $g(x)(k)=x(k)$; if not, let $k'$ be the starting position of the next expected occurrence of $v_1$ in $x$, and set $g(x)(k)=w_{i_0,j_0}(|w_{i_0,j_0}|+k-k')$. Let $Z=g(X_V)$. Then $(Z,\sigma)$ is a subshift and $g$ is a factor map. By our definition, $L_{n+2n_1}(Z)=L_{n+2n_1}(W)=L_{n+2n_1}(X_W)$. 

For $x\in Z$ and $k\in \mathbb{Z}$, define $h(x)(k)=\Phi({x\rest{[k-n_1,k+n_1]}})$. Let $Y=h(Z)$. Then $(Y,\sigma)$ is a subshift and $h$ is a factor map. By our definition, $L_n(Y)=L_n(X)$. It also follows that there exists $y\in Y$ such that $y$ does not have a period $t$ for any $t\le |w_{i_0,1}|+|w_{i_0,2}|$. 

By Theorem 1.5 of \cite{GZ19}, the maximal equicontinuous factor of $X_V$ is a finite cycle of length $p$, where $p$ is the maximum such that for sufficiently large $i$, $p$ divides both $|v_i|$ and $|v_i|+|w_{i_0,j_0}|$. It follows that $p$ is a factor of $|w_{i_0,j_0}|$. However, since $Y$, a factor of $X_V$, contains an element which does not have a period $t$ for any $t\leq |w_{i_0,j_0}|$, we conclude that $Y$ is an infinite set. By the main theorem of \cite{GZ20}, any nontrivial factor of $X_V$ is conjugate to $X_V$. Thus $Y$ is conjugate to $X_V$. This finishes the proof of the main conclusion of the theorem.

Suppose $(X,\sigma)$ is totally transitive. We define $W$, $\{w_{i,j}\}_{i\geq 0,1\le j\le2}$, $n_1$, $i_0$, and $i_1$ as before. We claim that $|w_{i_0,1}|$ and $|w_{i_0,2}|$ are relatively prime. To see this, let $a=\gcd(|w_{i_0,1}|, |w_{i_0,2}|)$ and assume $a>1$. Then by property (2), the set of all $x\in X_W$ such that there exists an expected occurrence of $w_{i_0,1}$ or $w_{i_0,2}$ starting at  some multiple of $a$ is a clopen, $\sigma^a$-invariant, proper subset of $X$, contradicting the assumption that $(X,\sigma)$ is totally transitive. 

Let $p=|w_{i_0,j_0}|$ and $q=|w_{i_0,3-j_0}|$. Since $p, q$ are relatively prime, we can find a positive integer $m$ such that $|w_{i_0,j_0}w_{i_1,1}|+(m+1)p$ and $p-q$ are relatively prime. We inductively define a rank-1 generating sequence as follows. First let 
$$ v_1=w_{i_0,j_0}w_{i_1,1}(w_{i_0,j_0})^m. $$
For $i\ge1$, if $v_i$ has been defined such that $|v_i|+p$ and $|v_i|+q$ are relatively prime, then let $v_{i+1}$ be defined to satisfy the following properties:
\begin{enumerate}
{\item[(i)] $v_{i+1}$ is built from $v_i$ and the spacer parameters are only selected from $\{p,q\}$;
\item[(ii)] for any $0\le j<i$, there exist $k_1<k_2$ such that $k_2-k_1-j$ is a multiple of $i$, and $k_1,k_2$ are the starting positions of expected occurrences of $v_i$ in $v_{i+1}$;
\item[(iii)] $|v_{i+1}|+p$ and $|v_{i+1}|+q$ are relatively prime.
}
\end{enumerate}
Let $V=\lim_i v_i$. Then $V$ is a minimal aperiodic infinite rank-$1$ word. By (ii), $(X_V,\sigma)$ is totally transitive. The rest of the argument is identical to the above proof.
\end{proof}

\begin{corollary}\label{cor:new7.3} The set of all minimal subshifts conjugate to a rank-$1$ subshift is dense in $\overline{\mathcal{T}'_2}$ and $\overline{\mathcal{TT}'_2}$.
\end{corollary}

\begin{proof} By Theorems 1.3 and 1.4 of \cite{PS}, a generic subshift in $\overline{\mathcal{T}'_2}$ or $\overline{\mathcal{TT}'_2}$ is minimal and has topological rank $2$. Thus the conclusion follows from Theorem~\ref{thm:new7.2}.
\end{proof}

\begin{theorem} The set of all minimal subshifts conjugate to a rank-$1$ subshift is not generic in either $\mathcal{S}'_2$, $\overline{\mathcal{T}'_2}$ or $\overline{\mathcal{TT}'_2}$. Moreover, it is not $G_\delta$ in either $\overline{\mathcal{T}'_2}$ or $\overline{\mathcal{TT}'_2}$.
\end{theorem}

\begin{proof} By Corollary 4.9 of \cite{PS}, the set of all minimal subshifts is nowhere dense in $\mathcal{S}'_2$. 

By Theorem 1.3 of \cite{PS}, a generic subshift in $\overline{\mathcal{T}'_2}$ is a regular Toeplitz subshift which factors onto the universal odometer. In contrast, by Theorem 1.5 of \cite{GZ19}, the maximal equicontinuous factor of a rank-$1$ subshift is finite. Hence the set of all minimal subshifts conjugate to a rank-$1$ subshift is not generic in $\overline{\mathcal{T}'_2}$. Since it is dense in $\overline{\mathcal{T}'_2}$ by Corollary~\ref{cor:new7.3}, it is not a $G_\delta$ in $\overline{\mathcal{T}'_2}$.

By Theorem 1.4 of \cite{PS}, a generic subshift in $\overline{\mathcal{TT}'_2}$ is topologically mixing. In contrast, by Theorem 1.3 of \cite{GZ19}, a minimal rank-$1$ subshift is never topologically mixing. Hence the set of all minimal subshifts conjugate to a rank-$1$ subshift is not generic in $\overline{\mathcal{TT}'_2}$. Since it is dense in $\overline{\mathcal{TT}'_2}$ by Corollary~\ref{cor:new7.3}, it is not a $G_\delta$ in $\overline{\mathcal{TT}'_2}$.
\end{proof}

\begin{theorem} The set of all minimal subshifts conjugate to a subshift of symbolic rank $\leq 2$ is generic in $\overline{\mathcal{T}'_2}$ and $\overline{\mathcal{TT}'_2}$.
\end{theorem}

\begin{proof} By Theorems 1.3 and 1.4 of \cite{PS}, a generic subshift in $\overline{\mathcal{T}'_2}$ or $\overline{\mathcal{TT}'_2}$ is minimal and has topological rank $2$. Thus the conclusion follows from Theorem~\ref{thm:6.9}.
\end{proof}

\section{Factors of subshifts of finite symbolic rank\label{sec:7}}

By results of \cite{GH*}, \cite{Es} and our Corollary~\ref{cor:6.8} and Theorem~\ref{thm:6.9}, a Cantor system that is a factor of a minimal subshift of finite symbolic rank is conjugate to a minimal subshift of finite symbolic rank.

In this final section of the paper we prove some further results about factors of minimal subshifts of finite symbolic rank, and in particular about odometer factors and non-Cantor factors of mininal subshifts of finite symbolic rank. In the following we first show that for any $N\geq 1$, there exist minimal subshifts of finite symbolic rank which are not factors of minimal subshifts of symbolic rank $\leq N$.

\begin{lemma}\label{lem:7.1} For any $N\geq 1$, there exist $m>N$ and a good rank-$m$ construction with associated rank-$m$ generating sequence $\{v_{i,j}\}_{i\geq 0,1\le j\le m}$ such that the following hold for all $i\geq 1$:
\begin{enumerate}
\item[(i)] for any $1\le j_1,j_2\le m$ with $j_1\neq j_2$, $v_{i,j_1}$ is not a subword of $v_{i,j_2}$;
\item[(ii)] for any $1\le j\le m$, there is a unique building of $v_{i+1,j}$ from \newline $\{v_{i,1},\dots, v_{i,m}\}$ whose spacer parameter is bounded by 0;
\item[(iii)] there is a positive integer $\ell\geq 1$ such that, given any two finite sequences $(j_1,j_2,\cdots,j_\ell)$ and $(j'_1,j'_2,\cdots,j'_\ell)$ of elements of $\{1,2,\cdots,m\}$, there is at most one element $w$ of 
$$\{v_{i,j}v_{i,j'}\,:\,1\le j,j'\le m\}\cup\{v_{i,j}\,:\,1\le j\le m\}$$ such that 
$$v_{i,j_1}v_{i,j_2}\cdots v_{i,j_\ell}wv_{i,j'_1}v_{i,j'_2}\cdots v_{i,j'_\ell}$$ is a subword of $V\triangleq \lim_{n}v_{n,1}$;
\item[(iv)] $X_V$ is minimal and $\symbrank(X_V)\geq N$.
\end{enumerate}
\end{lemma}

\begin{proof}
Let $(X,T)$ be a minimal Cantor system whose topological rank is $K<\infty$ where $K\geq 8N^2$. By Theorem~\ref{thm:6.9} there exist $k\leq K$ and a proper rank-$k$ construction of an infinite word $W$ such that $(X_W,\sigma)$ is conjugate to $(X,T)$. It also follows from the proof of Theorem~\ref{thm:6.9} that the spacer parameter of $W$ is bounded by $1$. Let $m=2k$. By the proof of Proposition~\ref{prop:6.3}, there exists an infinite word $V$ with a good rank-$m$ construction such that $X_W$ is a factor of $X_V$. Moreover, the spacer parameter of $V$ is also bounded by $1$ and so $X_V$ is minimal. By analyzing the proof of Proposition~\ref{prop:6.3}, we can see that this construction satisfies (i), (ii) and (iii). In fact, (i) and (ii) are explicit from the proof. For (iii) we can take $\ell$ to be larger than the lengths of all buildings of $v_{i+1,j}$ from $\{v_{i,1},\dots, v_{i,m}\}$ for $1\leq j\leq m$. Then (iii) follows from the argument for the goodness of the construction in the proof of Proposition~\ref{prop:6.3}.

It remains to verify that $\symbrank(X_V)\geq N$. Suppose $\symbrank(X_V)=n$. Then by Corollary~\ref{cor:6.8} we have $\toprank(X_V,\sigma)\leq 8n^2$. By \cite{Es}, $K=\toprank(X,T)=\toprank(X_W,\sigma)\leq \toprank(X_V,\sigma)\leq 8n^2$. Since $K\geq 8N^2$, we have $n\geq N$. 
\end{proof}

\begin{proposition}\label{prop:7.2} For any $N\geq 1$, there exists a minimal subshift $X_V$ which is not a factor of any minimal subshift of symbolic rank $\leq N$. In particular, $X_V$ is not conjugate to any minimal subshift of symbolic rank $\leq N$.
\end{proposition}

\begin{proof}
By Lemma~\ref{lem:7.1} there is $m>4N^2+1$ and we have an infinite word $V$ which has a good rank-$m$ construction with associated rank-$m$ generating sequence $\{v_{i,j}\}_{i\geq 0,1\le j\le m}$ satisfying (i), (ii) and (iii) in Lemma~\ref{lem:7.1}, so that $X_V$ is minimal and $\symbrank(X_V)\geq 4N^2+1$. Assume toward a contradiction that $n\leq N$ and $W'$ has a proper rank-$n$ construction with bounded spacer parameter such that $X_V$ is a factor of $X_{W'}$.  

By Proposition~\ref{prop:6.3}, we have an infinite word $W$ which has a good rank-$2n$ construction with associated rank-$2n$ generating sequence $\{w_{p,q}\}_{p\geq 0,1\le q\le 2n}$ such that $X_W$ is minimal and $X_{W'}$ is a factor of $X_W$. Let $f$ be a factor map from $(X_W,\sigma)$ to $(X_V,\sigma)$. 

Let $k_1$ be a positive integer such that $1^{k_1}$ is not a subword of $W$. Let $k_2$ be a positive integer such that for any $x,y\in X_W$ and $k\in\mathbb{Z}$, whenever $x\rest [k-k_2, k+k_2]=y\rest [k-k_2,k+k_2]$, we have $f(x)(k)=f(y)(k)$. Let $r\ge 1$ so that $\min_{1\le j\le m}|v_{r,j}|\gg k_1+2k_2$. Let $\ell\geq 1$ be given by (iii) in Lemma~\ref{lem:7.1}, that is, for any two finite sequences $(j_1,j_2,\cdots,j_\ell)$ and $(j'_1,j'_2,\cdots,j'_\ell)$ of elements of $\{1,2,\cdots,m\}$, there is at most one element $w$ of 
$$\{v_{r,j}v_{r,j'}\,:\,1\le j,j'\le m\}\cup\{v_{r,j}\,:\,1\le j\le m\}$$ such that 
$$v_{r,j_1}v_{r,j_2}\cdots v_{r,j_\ell}wv_{r,j'_1}v_{r,j'_2}\cdots v_{r,j'_\ell}$$ is a subword of $V$. We can also find $s_0\ge1$ so that 
$$\displaystyle\frac{\min_{1\le q\le 2n}|w_{s_0,q}|-2k_2}{\max_{1\le j\le m}|v_{r,j}|}>\ell+2.$$ 

We claim that for any $s\geq s_0$, $1\le q,q'\le 2n$, $a\geq 0$, and $x\in X_W$, if $w_{s,q}1^aw_{s,q'}$ occurs in $x$, where the demonstrated occurrences of $w_{s,q}$ and $w_{s,q'}$ are expected, then $a$ is determined by $q$ and $q'$ only (and in particular $a$ does not depend on $x$). To see this, let $k$ be the starting position of the assumed occurrence of $w_{s,q}1^aw_{s,q'}$ in $x$, and let $k'$ be the starting position of the demonstrated occurrence of $w_{s,q'}$. Then $f(x)\rest [k+k_2, k+|w_{s,q}|-k_2-1]$ and $f(x)\rest [k'+k_2,k'+|w_{s,q'}|-k_2-1]$ are determined only by $w_{s,q}$ and $w_{s,q'}$ by our assumption, and since $s\geq s_0$, each of them contains a subword of the form $v_{r,j_1}v_{r,j_2}\dots v_{r,j_\ell}$. Since $a<k_1$ and $\min_{1\leq j\leq m}|v_{r,j}|\gg k_1+2k_2$, we get that $f(x)\rest[k+k_2,k'+|w_{s,q'}|-k_2-1]$ contains a subword of the form 
$$ v_{r,j_1}v_{r,j_2}\cdots v_{r,j_\ell}wv_{r,j_1'}v_{r,j_2'}\cdots v_{r,j'_\ell} $$
where $f(x)\rest[k+k_2,k+|w_{s,q}|-k_2-1]$ contains the part $v_{r,j_1}\cdots v_{r,j_\ell}$, $f(x)\rest [k'+k_2,k'+|w_{s,q'}-k_2-1]$ contains the part $v_{r,j_1'}\cdots v_{r, j_{\ell}'}$, and $w$ is either of the form $v_{r,j}$ for some $1\leq j\leq m$ or of the form $v_{r,j}v_{r,j'}$ for $1\leq j,j'\leq m$. By our assumption, there is a unique such $w$, which implies that there is a unique $a$ by considering $|w|$.

By telescoping we may assume that the claim holds for any $s\geq 1$. We may also assume that $|w_{1,q}|\gg 2k_2+k_1+k_0$ for $1\le q\le 2n$, where $k_0$ is such that $1^{k_0}$ is not a subword of $V$. For any finite word $u$, let $\tilde{u}\in\mathcal{F}$ be the unique subword of $u$ such that $u=1^a\tilde{u}1^b$ for some nonnegative integers $a, b$. Now we define a set $T_s$ of finite words in $\mathcal{F}$ for all $s\geq 0$ as follows. For any $s\ge1$ and $1\le q,q'\le 2n$, if there are $x\in X_W$, $k\in \mathbb{Z}$, and $a\geq 0$ such that the word $w_{s,q}1^aw_{s,q'}$ occurs in $x$, where the demonstrated occurrences of $w_{s,q}$ and $w_{s,q'}$ are expected, then define a word $u_{s,q,q'}=\tilde{u}$ where $u=f(x)\rest [k+k_2,k+|w_{s,q}|+a+k_2-1]$. Let $T_s$ be the set of all $u_{s,q,q'}$ thus obtained for $s\geq 1$ and $1\leq q,q'\leq 2n$.  Let $T_0=\{0\}$. Then the sequence $\{T_s\}_{s\geq 0}$ satisfies the hypotheses of Proposition~\ref{prop:4.1}; in particular, every element of $T_{s+1}$ is built from $T_s$. Also, $|T_s|\leq 4n^2$. By Proposition~\ref{prop:4.1} we obtain a rank-$4n^2$ construction of an infinite word $V'$. Since each $u_{s,q,q'}$ is a subword of $V$, we have that $X_{V'}\subseteq X_V$. By the minimality of $X_V$, we have $X_{V'}=X_V$, and thus $X_V$ has symbolic rank $\leq 4n^2\leq 4N^2$, contradicting $\symbrank(X_V)\geq 4N^2+1$.
\end{proof}

In a sense, we separate the topological rank (or alphabet rank) complexity of a subshift into two parts: symbolic rank and spacer parameters. This proposition together with Proposition~\ref{prop:6.11} shows that both of these two parts are nontrivial.

Next we show that an infinite subshift factor of a minimal subshift of finite symbolic rank is not just conjugate to a subshift of finite symbolic rank -- it is itself a subshift of finite symbolic rank. This is a technical improvement of the result we mentioned at the beginning of this section. The proof  of this result is similar to the one for the above proposition.

\begin{theorem}\label{thm:new} Let $X$ be minimal subshift of finite symbolic rank and $Y$ be an infinite subshift that is a factor of $X$. Then $Y$ has finite symbolic rank, i.e., there is an infinite word $V$ with a finite rank construction such that $Y=X_V$.
\end{theorem}

\begin{proof} By Proposition~\ref{prop:6.3} we may assume that $X=X_W$ where $W$ has a good rank-$n$ construction for some $n\geq 2$, with associated rank-$n$ generating sequence $\{w_{p,q}\}_{p\geq 0, 1\leq q\leq n}$. Let $f$ be a factor map from $(X_W,\sigma)$ to $(Y,\sigma)$. 

Let $k_1$ be a positive integer such that $1^{k_1}$ is not a subword of $W$. Let $k_2$ be a positive integer such that for any $x,y\in X_W$ and $k\in\mathbb{Z}$, whenever $x\rest[k-k_2,k+k_2]=y\rest[k-k_2,k+k_2]$, we have $f(x)(k)=f(y)(k)$. $Y$ is an infinite minimal subshift, let $k_3$ be a positive integer such that $1^{k_3}$ is not a subword of $x$ for any $x\in Y$. Without loss of generality we may assume $|w_{1,q}|\gg 2k_2+k_1+k_3$ for all $1\leq q\leq n$. 

Similar to the above proof, for each $p\geq 1$, if the word $w_{p,q}1^sw_{p,q'}$ occurs in some $x\in X_W$ at position $k\in\mathbb{Z}$, where the demonstrated occurrences of $w_{p,q}$ and $w_{p,q'}$ are expected, we define a word $u_{p,q,q',s}=\tilde{u}$ where 
$$u=f(x)\rest[k+k_2, k+|w_{s,q}|+s+k_2-1].$$
Then it is clear that every $y\in Y$ is built from 
$$T_p=\{u_{p,q,q',s}\,:\, 1\leq q,q'\leq n, 0\leq s<k_1\}.$$ By Proposition~\ref{prop:4.1} we obtain a rank-$n^2k_1$ construction of an infinite word $V$ such that $X_V\subseteq Y$. By the minimality of $Y$ we must have $X_V=Y$, and thus $Y$ has finite symbolic rank.
\end{proof}

A curious example is when $V$ is an infinite rank-$1$ word and $\varphi: X_V\to Y$ is the conjugacy map  defined by the substitution $0\mapsto 1$ and $1\mapsto 0$. $Y$ is in general no longer a rank-$1$ subshift but it has finite symbolic rank.

The above theorem has the following immediate corollary.

\begin{corollary}\label{cor:7.4} Let $n\geq 2$ and let $X$ be a minimal subshift of topological rank $n\geq 2$. Then $X$ has finite symbolic rank.
\end{corollary}

\begin{proof} By Theorem~\ref{thm:6.9} $X$ is conjugate to a minimal subshift of finite symbolic rank. Thus $X$ has finite symbolic rank by Theorem~\ref{thm:new}.
\end{proof}

Next we show that any infinite odometer is the maximal equicontinuous factor of a minimal subshift of symbolic rank $2$. This is in contrast with the result in \cite{GZ19} that any equicontinuous factor of a rank-$1$ subshift is finite.

We use the following fact, which is folklore.

\begin{lemma}\label{lem:7.3} Let $(X,T)$ and $(Y,S)$ be topological dynamical systems and let $f$ be a factor map from $(X,T)$ to $(Y,S)$. Suppose $(Y,S)$ is equicontinuous and suppose for all $x_1, x_2\in X$, if $f(x_1)=f(x_2)$ then $x_1,x_2$ are proximal. Then $(Y,S)$ is the maximal equicontinuous factor of $(X,T)$.
\end{lemma}



\begin{theorem}\label{thm:7.4}  For any infinite odometer $(Y,S)$, there exists a minimal subshift $X_V$ of symbolic rank $2$ such that $(Y,S)$ is the maximal equicontinuous factor of $(X_V,\sigma)$.
\end{theorem}

\begin{proof}  We inductively define two sequences $\{p_i, q_i\}_{i\geq 0}$ of positive integers as follows. Let $p_0=q_0=1$. For $i\geq 0$, let $p_{i+1}=2p_i+2q_i$ and $q_{i+1}=2p_i+q_i$. It is easy to see that for any $i\geq 0$, $q_i$ is odd and $p_i,q_i$ are relatively prime. 

Let $B=(W,E,\preceq)$ be a simple Bratteli diagram associated to $(Y,S)$ such that $|W_i|=1$ and $a_{i+1}\triangleq |E_{i+1}|>1$ for all $i\ge 0$. By telescoping, we may assume $a_i\gg p_i+q_i$ for any $i\ge 1$. Consider the following proper rank-$2$ construction:
$$\begin{array}{rclrcl} v_{0,1}&\!\!\!\!=\!\!\!\!&v_{0,2}=0, & &\!\!\!\! & \\ 
v_{1,1}&\!\!\!\!=\!\!\!\!&0^{a_1}1^{2a_1}0^{a_1},& v_{1,2}&\!\!\!\!=\!\!\!\!&0^{a_1}1^{a_1}0^{a_1}, \\ 
v_{i+1,1}&\!\!\!\!=\!\!\!\!&{v_{i,1}}^{a_{i+1}}{v_{i,2}}^{2a_{i+1}}{v_{i,1}}^{a_{i+1}},& v_{i+1,2}&\!\!\!\!=\!\!\!\!&{v_{i,1}}^{a_{i+1}}{v_{i,2}}^{a_{i+1}}{v_{i,1}}^{a_{i+1}}, \mbox{ for $i\ge1$.}
\end{array}
$$
It is easy to see that the subsequence $\{i_k\}_{k\geq 0}$ where $i_0=0$ and $i_k=k+1$ for $k\geq 1$ gives a telescoped construction that is good and hence has unique readability. Also, for any $n\geq 1$, $|v_{n,1}|=p_n\prod_{i=1}^na_i$ and $|v_{n,2}|=q_n\prod_{i=1}^na_i$. For notational simplicity let $A_n=\prod_{i=1}^na_i$ for all $n\geq 1$ and let $A_0=1$. Let $V=\lim_{n}v_{n,1}$.

For each $i\geq 1$, enumerate the elements of $E_i$ in the $\preceq$-increasing order as $e_{i,1},\dots, e_{i,a_i}$.  Define $f:X_V\rightarrow X_B$ by letting $f(x)(i)=e_{i+1,j}$ if there exists an expected occurrence of $v_{i+1,1}$ in $x$ starting at position $k\in\mathbb{Z}$ such that for some $\ell\in\mathbb{Z}$ we have that $1\leq j\leq a_{i+1}$ satisfies 
$$ (j-1)A_i\leq k+\ell A_{i+1}<jA_i. $$
$f$ is well-defined because $|v_{i+1,1}|$ and $|v_{i+1,2}|$ are both multiples of $A_{i+1}$, and thus for any two expected occurrences of $v_{i+1,1}$ in $x$, their starting positions differ by a multiple of $A_{i+1}$. It is clear that $f$ is a factor map from $(X_V,\sigma)$ to $(X_B,\lambda_B)$. 


By Lemma~\ref{lem:7.3}, in order to complete the proof, it suffices to show that for any $x,y\in X_V$, if $f(x)=f(y)$ then $x,y$ are proximal. Toward a contradiction, assume $x,y$ are not proximal but $f(x)=f(y)$. Thus there exists $n\geq 1$ such that no $k\in\mathbb{Z}$ is the starting position of both an expected occurrence of $v_{n,1}$ in $x$ and one in $y$. Let $n_0$ be the least such $n$.

On the other hand, from the assumption $f(x)=f(y)$, we can verify by induction that for all $n\geq 0$, if $k_1$ is the starting position of an expected occurrence of $v_{n+1,1}$ or $v_{n+1,2}$ in $x$ and $k_2$ is the starting position of an expected occurrence of $v_{n+1,1}$ or $v_{n+1,2}$ in $y$, then $k_1-k_2$ is a multiple of $A_{n+1}$. 


We claim that there exist no $k<h$ such that $h-k=tA_{n_0+1}$ for some $1\le t<p_{n_0+1}$, $h$ is the starting position of at least $p_{n_0+1}$ many consecutive expected occurrences of $v_{n_0+1,2}$ in $x$ (or $y$), and $k$ is the starting position of at least $q_{n_0+1}$ many consecutive expected occurrences of $v_{n_0+1,1}$ in $y$ (or $x$, respectively).


If not, then from the property that $p_{n_0+1}$ and $q_{n_0+1}$ are relatively prime, we can get positive integers $a< q_{n_0+1}$ and $b< p_{n_0+1}$ such that $t=ap_{n_0+1}-bq_{n_0+1}$. Then $k+a|v_{n_0+1,1}|=h+b|v_{n_0+1,2}|$. This is the starting position of an expected occurrence of $v_{n_0+1,1}$ in $y$ (or $x$), while at the same time it is also the starting position of an expected occurrence of $v_{n_0+1,2}$ in $x$ (or $y$, respectively). Thus it is the starting position of an expected occurrence of $v_{n_0,1}$ in both $x$ and $y$, contradicting our definition of $n_0$. 


Now let $P$ be the $(n_0+2)$-th layer of the reading of $x$, that is, $(k,j)\in P$ iff there is an expected occurrence of $v_{n_0+2,j}$ in $x$; let $Q$ be the $(n_0+2)$-th layer of the reading of $y$. Suppose $(k,j)\in P$ where $j=1$ or $2$. Consider the positions from $k+a_{n_0+2}|v_{n_0+1,1}|$ to $k+a_{n_0+2}|v_{n_0+1,1}|+(3-j)a_{n_0+2}|v_{n_0+1,2}|$. If one of these positions is the starting position of an expected occurrence of $v_{n_0+2,1}$ or $v_{n_0+2,2}$ in $y$, then from $a_{n_0+2}\gg p_{n_0+2}+q_{n_0+2}$, we get a contradiction to the above claim. So these positions must be contained in the same expected occurrence of $v_{n_0+2,1}$ or $v_{n_0+2,2}$ in $y$, which gives us a unique $(k',j')\in Q$. It follows from the above claim and the assumption $a_{n_0+2}>>p_{n_0+2}+q_{n_0+2}$ that $j'=j$ and $|k-k'|<\frac{1}{4}|v_{n_0+2,2}|$. Let $m=k-k'$. Applying this to all $(k,j)\in P$, we obtain corresponding $(k',j)\in Q$ and $m=k-k'$. Clearly $m$ is constant, which implies that $y=\sigma^m(x)$ and that $f(x)=f(y)$ is periodic, a contradiction.
\end{proof}

Finally we consider non-Cantor factors of subshifts of finite symbolic rank. By a combination of existing research, we can see that any irrational rotation is the maximal equicontinuous factor of a minimal subshift of symbolic rank $2$. In fact the symbolic rank-$2$ subshifts are generated by the Sturmian sequences that are symbolic representations of irrational rotations (for details, see e.g. \cite[Section 6.1.2]{Ar}). In \cite{GJJLLSW} it was shown that all Sturmian sequences have a proper rank-$2$ construction. As noted in \cite{DDM2000}, it follows from the work of \cite{HM} that for any irrational number $0<\alpha<1$, there is a Sturmian sequence $V_\alpha$ and a factor map $\theta$ from $(X_{V_\alpha},\sigma)$ to $(\mathbb{T}, +\alpha)$ such that $\theta$ is injective on a comeager subset of $X_{V_\alpha}$. By a well-known criterion (e.g. \cite[Proposition 1.1]{Wi}), $(\mathbb{T},+\alpha)$ is the maximal equicontinuous factor of $(X_{V_\alpha},\sigma)$. Conversely, since any Sturmian sequence is $V_\alpha$ for some irrational $0<\alpha<1$ (see e.g. \cite[Theorem 6.4.22]{Ar}), the maximal equicontinuous factor of a subshift generated by a Sturmian sequence is an irrational rotation.

\end{document}